\documentclass[11pt,reqno]{amsart}
\numberwithin{equation}{section}
\usepackage{hyperref}
\hypersetup{nesting=true,debug=true,naturalnames=true}
\usepackage{graphicx,amssymb,upref}
\usepackage[usenames,dvipsnames]{pstricks}
\usepackage{epsfig}
\usepackage{color}
\usepackage[utf8]{inputenc}
\usepackage[T1]{fontenc}
\usepackage{amsmath, amsthm}
\usepackage[english]{babel}
\usepackage{amssymb}
\usepackage{latexsym}
\usepackage{graphicx,amssymb,upref}
\usepackage[all]{xy}
\usepackage{tikz}
\usetikzlibrary{matrix} 
\usepackage{mathtools}
\usepackage{pdfpages}
\usepackage{multirow}
\usepackage{comment}
\allowdisplaybreaks
\usepackage{bm}

\def\aaa{{\mathfrak a}}

\date{\today}

\newcounter{blist} 
\newenvironment{blist}{\begin{list}{(\alph{blist})}{ 
			\usecounter{blist}\leftmargin2.5em\labelwidth2em\labelsep0.5em 
			\topsep0.6ex 
			\parsep0.3ex plus0.2ex minus0.1ex}}{\end{list}} 

\newtheorem{theorem}{Theorem}[section]
\newtheorem{lemma}[theorem]{Lemma}
\newtheorem{proposition}[theorem]{Proposition}
\newtheorem{corollary}[theorem]{Corollary}
\theoremstyle{definition}
\newtheorem{definition}[theorem]{Definition}
\newtheorem{example}[theorem]{Example}

\newtheorem{remark}[theorem]{Remark}

\def\Tan#1#2{T_{#1}{#2}}
\newcommand{\np}{\raisebox{0.2ex}{\scalebox{0.6}{$\,\bullet\,$}}}

\newcommand{\paff}{\mathfrak{p}}

\begin{document}
	
	\title[Affinization: Poisson algebras]{Affinization of algebraic structures:  Poisson algebras}

	\author[Brzezi\'nski]{Tomasz Brzezi\'nski}
	
	\address{
		Department of Mathematics, Swansea University,
		Fabian Way,
		Swansea SA1 8EN, U.K.\ \newline \indent
		Faculty of Mathematics, University of Bia{\l}ystok, K.\ Cio{\l}kowskiego  1M,
		15-245 Bia\-{\l}ys\-tok, Poland}
	
	\email{T.Brzezinski@swansea.ac.uk}

	\author[Radziszewski]{Krzysztof Radziszewski}
	
	\address{Doctoral School \&
		Faculty of Mathematics, University of Bia{\l}ystok, K.\ Cio{\l}kowskiego,
		15-245 Bia\-{\l}ys\-tok, Poland}
	
	\email{K.Radziszewski@uwb.edu.pl}
	
	\author[Ramos Pérez]{Brais Ramos Pérez}
	
	\address{
		Faculty of Science, Department of Mathematics, University of Oviedo, 33007 Oviedo, Spain.}
	
	\email{ramosbrais@uniovi.es}
	\thanks{\emph{Corresponding author.} Tomasz Brzezi\'nski, \texttt{T.Brzezinski@swansea.ac.uk}}
	
	\begin{abstract}
		An affinization of the notion of a Poisson algebra is presented.  This is termed a Poisson affgebra and consists of an affine space together with an associative bi-affine multiplication and a bi-affine Lie bracket  that acts as an affine derivation for the associative product. The constructive relation between Poisson affgebras and Poisson algebras is described and several low-dimensional examples are studied in detail. 
	\end{abstract} 
	
	\vspace{0.2cm}

	
	
	\maketitle
	%-------------------------------------------------------------------------
	\section{Introduction}\label{sec.intro}
	The recently formalised \textit{affinization programme of algebraic structures} has two guiding objectives. The first one is to find a common framework that connects rings with braces introduced in \cite{Rum:bra}, \cite{CedJes:bra} as tools for solving and classifying solutions of the Yang-Baxter equation; see \cite{Brz:tru}. The second one is to establish an \textit{absolute} (i.e.\ vector space independent) algebraic scheme for the Tulczyjew programme of a frame-independent formulation of analytical mechanics \cite{Tul:fra}. The present text is intended as a contribution to this aim.
	
	The Tulczyjew programme involves replacing vector by affine spaces (hence, the \textit{affinization}). Its geometric side is covered comprehensively by the \textit{AV-geometry}, see e.g.\ \cite{GraGra:av1}, \cite{GraGra:av2}, although this approach still relies on the vector space dependent definition of an affine space. Likewise, the notion of a \textit{Lie affgebra} that arose through the AV-geometry \cite{GGU} relies on the globally fixed choice of a vector space. That vector space independent formulation, and indeed generalisation of Lie affgebras is possible, has been shown in \cite{BRZ3}. In this article we intend to develop vector space independent or absolute framework for Poisson structures on affine spaces or to define and study \textit{Poisson affgebras}.
	
	The paper is organized as follows. In Section~\ref{sec.pre} we recall algebraic structures which are necessary for defining Poisson affgebras. These include in particular heaps and the absolute formulation of affine spaces in Section~\ref{ssec.heap}, Lie affgebras and their relation to Lie algebras in Section~\ref{ssec.Lie}, and associative affgebras and their relation to extensions of associative algebras by homothetic data in Section~\ref{ssec.ass}. This part is a review of known results except for the introduction of the concept of a \textbf{quasi-commutative affgebra}, as an affgebra that captures commutativity of the product of its algebra fibres.
	
	As the compatibility between the Lie and associative products in a Poisson algebra is provided by the condition that the former is a derivation of the latter, in Section~\ref{sec.der} we propose an affinization of derivations, introducing the notion of an \textbf{affine derivation} or a \textbf{deriffation} on an associative affgebra in Definition~\ref{def.deriffation}. This is designed in such a way that its linearization to any fibre is a derivation of the algebra product of this fibre. 
	
	Section~\ref{sec.Poisson} contains the main results of the paper. Specifically, the notion of a \textbf{Poisson affgebra} is introduced and its relation to Poisson algebras is studied. In particular it is shown in Theorem~\ref{main.th.poisson} that any Poisson affgebra arises from a Poisson algebra with a homothetic datum (that controls the affinization of its associative product) and a generalized derivation of its underlying Lie algebra structure whose linear endomorphisms involved are derivations of the associative product. We illustrate this relation by the example of the Darboux-type Poisson bracket on the algebra of real polynomials in two variables.   The connection between Poisson algebras and affgebras allows one to relate homomorphisms of Poisson affgebras with suitable homomorphisms of Poisson algebras (see Theorem~\ref{hom.Poisson.aff}) thus allowing for potential classification of Poisson affgebras (relative to the classification of Poisson algebras). 
	
	Further examples and classification of Poisson affgebras are presented in Section~\ref{sec.ex}. In particular, in Section~\ref{clas1dim} we describe the classification of one-dimensional Poisson affgebras arriving at five families of such affgebras. In Section~\ref{ssec.2-dim} we look at the classification of two-dimensional Poisson affgebras, illustrating it by extensive studies of an example in which the linear Poisson bracket is trivial and an example with a non-trivial linear Poisson bracket.

	Throughout  the text, $\mathbb{F}$ denotes a generic field.
	
	\section{Preliminaries}\label{sec.pre}
	In this section we present the necessary preliminary notions. First we recall the notion of a Poisson algebra. As it combines a Lie and an associative structure on a common vector space, in the following subsections we recall the vector space independent definition of an affine space and then discuss Lie brackets and associative multiplications on such absolute affine spaces.
	\subsection{Poisson algebras}
	The following notion was first introduced in \cite{VinKra:Ham}, \cite{Lic:var} and \cite{Bra:alg}.
	\begin{definition}\label{def.Poi.alg}
		A \textbf{Poisson algebra} over $\mathbb{F}$ is an $\mathbb{F}$-vector space $P$ equipped with two bilinear operations $[-,-]\colon P\times P\rightarrow P$ and $-\cdot- \colon P\times P\rightarrow P$ such that:
		\begin{enumerate}
			\item[(i)] $(P,[-,-])$ is a Lie algebra,
			\item[(ii)] $(P, \cdot)$ is an associative algebra,
		\end{enumerate}
		and the following condition is satisfied
		\begin{equation}\label{Leib.rule}
			[x\cdot y,z]=x\cdot [y,z]+[x,z]\cdot y \quad \textnormal{(the Leibniz rule)}
		\end{equation}
		for all $x,y,z\in P$. In other words, \eqref{Leib.rule} is equivalent to the fact that the maps $[-,z]\colon P\rightarrow P$ are derivations on $(P, \cdot)$ for all $z\in P$. It is worth noting that the antisymmetry of the Lie bracket $[-,-]$ implies that the maps $[z,-]$ are also derivations on $(P,\cdot)$ for all $z\in P$. 
		
		An $\mathbb{F}$-linear map $f\colon P\rightarrow P'$ is a {\bf morphism of Poisson algebras} if 
		\begin{gather*}%\label{morPois.alg}
			f(x\cdot y)=f(x)\cdot'f(y),\qquad f([x,y])=[f(x),f(y)]'
		\end{gather*}
		for all $x,y\in P$ or, in other words, if $f$ preserves the associative product and the Lie bracket. Let us denote by ${\sf Poiss}$ the category of Poisson algebras.
	\end{definition}
	
	Poisson algebras naturally occur in Hamiltonian mechanics, from which they originated (see \cite{VinKra:Ham}, \cite{Lic:var}).  A prototypical example is provided by the algebra of smooth functions $C^{\infty}(\mathbb{R}^{2n})$, which is an associative algebra with the pointwise product, and a Lie algebra with the bracket defined as follows:
	\begin{gather}\label{darboux}[f,g]=\sum_{i=1}^{n}\left(\frac{\partial f}{\partial x_{i}}\frac{\partial g}{\partial y_{i}}-\frac{\partial f}{\partial y_{i}}\frac{\partial g}{\partial x_{i}}\right),\end{gather}
	known as the \textit{Darboux bracket}. More generally, the algebra of smooth functions on a cotangent bundle $T^*M$  of a smooth manifold $M$ (i.e.\ the algebra of smooth functions on a phase space for the motion of a particle on the manifold $M$) is a Poisson algebra with the bracket that is the Darboux bracket in local coordinates. Note that $C^{\infty}(T^*M)$ is a commutative algebra, and often authors include the commutativity requirement in the definition of a Poisson algebra. In Definition~\ref{def.Poi.alg} the commutativity of multiplication is not assumed in order to maintain full algebraic generality.
	
	\subsection{Heaps and affine spaces }\label{ssec.heap}
	The following notion was introduced in \cite{Pru:the} and \cite{Bae:ein}.
	\begin{definition}
		A \textbf{heap} is a set $X$ together with a ternary operation $\langle-,-,-\rangle\colon X\times X\times X\rightarrow X$ which satisfies the following conditions:
		\begin{subequations}\label{heapcond}
			\begin{align}\label{h_asoc}
				\langle\langle a,b,c\rangle,d,e\rangle= \langle a,b,\langle c,d,e\rangle\rangle\quad&\textnormal{(associativity),}\\
				\label{h_Malcev1}\langle a,b,b\rangle=a \quad& \textnormal{(Mal'cev identity 1),}\\
				\label{h_Malcev2}\langle a,a,b\rangle=b \quad&\textnormal{(Mal'cev identity 2),}
			\end{align}
		\end{subequations}
		for all $a,b,c,d,e\in X$.
		
		Let $(X,\langle -,-,-\rangle_{X})$ and $(Y,\langle -,-,-\rangle_{Y})$ be heaps. A map $f\colon X\rightarrow Y$ is called a \textbf{heap morphism} if $f$ preserves the heap ternary operation, i.e., $f(\langle a,b,c\rangle_{X})=\langle f(a),f(b),f(c)\rangle_{Y}$ for all $a,b,c\in X$.
		
		A heap $(X,\langle-,-,-\rangle)$ is said to be \textbf{abelian} when
		\begin{gather}\label{h_ab}
			\langle a,b,c\rangle=\langle c,b,a\rangle
		\end{gather}
		for all $a,b,c\in X$. Under this property, \eqref{h_Malcev1} is equivalent to \eqref{h_Malcev2}.
	\end{definition}
	As a consequence of conditions \eqref{heapcond} and \eqref{h_ab}, when a heap is abelian the distribution of the angled bracket does not matter, and hence we write $\langle a,b,c,d,e\rangle  $ for $\langle\langle a,b,c\rangle,d,e\rangle $. All heaps appearing in this paper are assumed to be abelian.
	
	It is worth noting that heaps are strongly related with groups. By fixing an element $o$ in a non-empty abelian heap $X$, the binary operation $x+y\coloneqq \langle x,o,y\rangle$ induces an abelian group structure on $X$, called the \textbf{retract} of $X$ at $o$ and denoted by $X_{o}$. All the retracts at different points of a given heap are mutually isomorphic by the {\it translation} map $\tau_{o}^{e}(a)\coloneqq \langle a,o,e\rangle\colon X_{o}\rightarrow X_{e}$. In the opposite direction, any abelian group $(G,+)$ determines a unique abelian heap structure on $G$ with the operation $\langle a,b,c\rangle\coloneqq a-b+c$. Every morphism of groups is also a morphism between the corresponding
	heaps, and a heap morphism $f\colon X\rightarrow Y$ uniquely defines a morphism of groups between the respective retracts $X_o\to Y_e$ by the formula $x \mapsto f(x)-f(o)$.
	
	In the subsequent definition, following \cite{BreBrz:hea} we recall the vector-free formulation of an affine space (see also \cite{Ost: aff}).
	\begin{definition}\label{affspace}
		An \textbf{$\mathbb{F}$-affine space} (or simply, \textbf{affine space}) is a triple $(\mathfrak{a},\langle -,-,-\rangle, -\triangleright_{-}-)$, where $(\mathfrak{a},\langle -,-,-\rangle)$ is a non-empty abelian heap and $-\triangleright_{-}-$ is a ternary map:
		\begin{align*}
			-\triangleright_{-}-\colon \mathbb{F}\times\mathfrak{a}\times\mathfrak{a}&\rightarrow\mathfrak{a}\\(\alpha,a,b)&\mapsto \alpha\triangleright_{a}b,
		\end{align*}
		satisfying the following conditions:
		\begin{enumerate}
			\item[(i)] For all $\alpha\in\mathbb{F}$ and $a\in \mathfrak{a}$, the map $\alpha\triangleright_{a}-\colon\mathfrak{a}\rightarrow \mathfrak{a}$ is a heap morphism.
			\item[(ii)] For all $a,b\in\mathfrak{a}$, the map $-\triangleright_{a}b\colon\mathbb{F}\rightarrow\mathfrak{a}$ is a heap morphism considering that the heap ternary operation in $\mathbb{F}$ is given by $\langle\alpha,\beta,\gamma\rangle_{\mathbb{F}}\coloneqq \alpha-\beta+\gamma$ for all $\alpha,\beta,\gamma\in\mathbb{F}$.
			\item[(iii)] For all $a\in\mathfrak{a}$, the map $-\triangleright_{a}-\colon\mathbb{F}\times\mathfrak{a}\rightarrow \mathfrak{a}$ is a left action of the multiplicative monoid $\mathbb{F}$, i.e., $1_{\mathbb{F}}\triangleright_{a}b=b$ and $\alpha\triangleright_{a}(\beta\triangleright_{a}b)=(\alpha\beta)\triangleright_{a}b$, for all $\alpha,\beta\in\mathbb{F}$ and for all $b\in\mathfrak{a}$.
			\item[(iv)] For all $a,b\in\mathfrak{a}$, $0_{\mathbb{F}}\triangleright_{a}b=a$.
			\item[(v)] (\textbf{Base change property})
			\[\alpha\triangleright_{a}b=\langle \alpha\triangleright_{c}b,\alpha\triangleright_{c}a,a\rangle\]
			for all $\alpha\in\mathbb{F}$ and for all $a,b,c\in \mathfrak{a}$.
		\end{enumerate}
		
		Let $(\mathfrak{a},\langle -,-,-\rangle_{\mathfrak{a}}, -\triangleright^{\mathfrak{a}}_{-}-)$ and $(\mathfrak{b},\langle -,-,-\rangle_{\mathfrak{b}}, -\triangleright^{\mathfrak{b}}_{-}-)$ be affine spaces. A map $f\colon \mathfrak{a}\rightarrow\mathfrak{b}$ is an \textbf{affine morphism} if $f$ is a heap morphism and also the equality
		\begin{equation*}%\label{affmap_cond}
			f(\alpha\triangleright^{\mathfrak{a}}_{a}b)=\alpha\triangleright^{\mathfrak{b}}_{f(a)}f(b)
		\end{equation*}
		holds, for all $a,b\in\mathfrak{a}$ and for all $\alpha\in\mathbb{F}$.
	\end{definition}
	%\begin{remark}\label{morph_aff_lin}
	Let $(\mathfrak{a},\langle -,-,-\rangle, -\triangleright_{-}-)$ be an affine space and let us fix $o\in\mathfrak{a}$. The retract of $\mathfrak{a}$ at $o$, $\mathfrak{a}_{o}$, is an $\mathbb{F}$-vector space with the external operation
	\[\alpha a\coloneqq\alpha\triangleright_{o}a,\]
	for all $\alpha\in\mathbb{F}$ and for all $a\in\mathfrak{a}.$ This vector space is called the \textbf{tangent space} of $\mathfrak{a}$ at $o$ (or simply, the \textbf{fibre} of $\mathfrak{a}$ at $o$),  and it is denoted by $T_{o}\mathfrak{a}$.
	In terms of the external operation in $T_{o}\mathfrak{a}$, the operation $-\triangleright_{-}-$ is given by
	\begin{equation*}%\label{relation_external_triangle}
		\alpha\triangleright_{a}b=(1-\alpha)a+\alpha b.
	\end{equation*}
	
	Moreover, $T_{o}\mathfrak{a}$ acts freely and transitively over $\mathfrak{a}$ via the following action
	\begin{align*}
		\psi\colon \mathfrak{a}\times T_{o}\mathfrak{a}\rightarrow\mathfrak{a}\qquad (a,b)\mapsto \psi(a,b)\coloneqq \langle a,o,b\rangle.
	\end{align*}
	
	Besides, fibres of $\mathfrak{a}$ at different points are mutually isomorphic (as vector spaces) \cite[Remark~2.5]{BRZ2} because, in this situation, the translation isomorphism $\tau_{o}^{e}$ is a linear map. Therefore, the {\bf dimension} of an affine space $\aaa$ is well-defined as the dimension (as a vector space) of any of its fibres:
	\[\operatorname{dim}(\aaa)\coloneqq \operatorname{dim}_{\mathbb{F}}(T_{o}\aaa)\]
	for any $o\in\aaa$.
	
	In addition, if $f\colon \mathfrak{a}\rightarrow \mathfrak{b}$ is an affine morphism, then it induces the unique linear map $\overrightarrow{f}$ between $T_{o}\mathfrak{a}$ and $T_{u}\mathfrak{b}$ defined by $\overrightarrow{f}(a)\coloneqq f(a)-f(o) = \langle f(a),f(o),u\rangle$, for all $a\in\mathfrak{a}$. 
	Conversely, if $\widehat{f}\colon T_{o}\mathfrak{a}\rightarrow T_{u}\mathfrak{b}$ is a linear map and we fix an element $b\in\mathfrak{b}$, then the map 
	\begin{align*}
		f\colon \mathfrak{a}\rightarrow\mathfrak{b}\qquad a\mapsto f(a)\coloneqq \widehat{f}(a)+b=\langle \widehat{f}(a),u,b\rangle
	\end{align*}
	is an affine map. All affine maps from $\mathfrak{a}$ to $\mathfrak{b}$ arise in this way. Note that $b = f(o)$.
	
	To make a connection with the traditional definition of an affine space as a set $\mathfrak{a}$ with a free and transitive action of a vector space $\overrightarrow{\mathfrak{a}}$, we can set explicitly $\overrightarrow{\mathfrak{a}}= \Tan{o}{\mathfrak{a}}$ (or any other vector space isomorphic to a fibre of $\mathfrak{a}$). The action is then given by $a+v \coloneqq \langle a,o,v\rangle$, for all $a\in \mathfrak{a}$ and $v\in T_o\mathfrak{a}$, and the vector from $a$ to $b$ in $\overrightarrow{\mathfrak{a}}$ is then $\overrightarrow{ab}\coloneqq\langle o,a,b\rangle$.
	%\end{remark}
	\begin{definition}
		An affine space $\mathfrak{a}=(\mathfrak{a},\mu_{\mathfrak{a}})$ is said to be an {\bf affgebra} if $\mathfrak{a}$ is endowed with a bi-affine morphism $\mu_{\mathfrak{a}}\colon\mathfrak{a}\times\mathfrak{a}\rightarrow\mathfrak{a}$, which is referred to as the affine product. 
		
		Since $\mu_{\mathfrak{a}}(a,-)$ and $\mu_{\mathfrak{a}}(-,a)$ are heap morphisms for all $a\in \mathfrak{a}$, the equalities
		\begin{subequations}\label{truss}
			\begin{gather}
				\label{truss1}\mu_{\mathfrak{a}}(a,\langle x,y,z\rangle)=\langle\mu_{\mathfrak{a}}(a,x),\mu_{\mathfrak{a}}(a,y),\mu_{\mathfrak{a}}(a,z)\rangle,\\
				\label{truss2}\mu_{\mathfrak{a}}(\langle x,y,z\rangle,a)=\langle\mu_{\mathfrak{a}}(x,a),\mu_{\mathfrak{a}}(y,a),\mu_{\mathfrak{a}}(z,a)\rangle
			\end{gather}
		\end{subequations}
		hold for all $x,y,z\in\mathfrak{a}$. Therefore, equalities \eqref{truss} mean that any affgebra $\mathfrak{a}$ leads to a \textbf{truss} \cite{BRZ1}.
	\end{definition}
	\subsection{Lie affgebras}\label{ssec.Lie}
	The following definition was proposed in \cite{BRZ3} as a homogeneous version of the notion defined in \cite{BRZ2}, which in turn was a vector space free or absolute formulation and extension of that of \cite{GGU}.
	\begin{definition}
		Let $\aaa$ be an affgebra with affine product $\{-,-\}\colon\aaa\times\aaa\rightarrow\aaa$. If $\{-,-\}$   satisfies the following conditions:
		\begin{blist}
			\item \textit{Affine antisymmetry}: for all $a,b\in \aaa$,
			\begin{equation*}%\label{antisym}
				\langle \{a,b\},\{a,a\}, \{b,a\}\rangle = \{b,b\};
			\end{equation*}
			\item \textit{Affine Jacobi identity}, that is, for all $a,b,c\in \aaa$,
			\begin{equation*}%\label{Jacobi}
				\langle \{a,\{b,c\}\},\{a,\{a,a\}\}, \{b,\{c,a\}\}, \{b,\{b,b\}\}, \{c,\{a,b\}\} \rangle = \{c,\{c,c\}\}
			\end{equation*}
		\end{blist}
		then $(\aaa,\{-,-\})$ is called a \textbf{Lie affgebra}.
		The affine product $\{-,-\}$ in a Lie affgebra is usually named an \textbf{affine Lie bracket}.
	\end{definition}
	In \cite{BRZ3} the connection between these structures and usual Lie algebras was described in detail. For our further considerations, the most relevant results will be those contained in \cite[Theorem~2.6]{BRZ3} and \cite[Theorem~3.1]{BRZ3}. The former shows that, for any Lie affgebra $(\aaa,\{-,-\})$, the underlying vector space $T_o\aaa$ is a Lie algebra with the Lie bracket defined by
	\begin{equation}\label{linearLiebr}
		[a,b]\coloneqq\{a,b\}-\{a,o\}+\{o,o\}-\{o,b\}=\langle\{a,b\},\{a,o\},\{o,o\},\{o,b\},o\rangle,
	\end{equation}
	that is, the product $[-,-]$ defined above is the linearization of the bi-affine product $\{-,-\}$ at the point $o\in \aaa$. Meanwhile, the latter serves us a recipe on how to construct a Lie affgebra from a given Lie algebra, say $(\mathfrak{g},[-,-])$, using an element $t\in \mathfrak{g}$ and a generalized derivation $(\lambda,-\mu,\lambda)$ of $\mathfrak{g}$ in the sense of \cite{LL}. More explicitly, 
	\begin{equation}\label{affine Lie bracket}
		\{a,b\}\coloneqq[a,b]+\mu(a)+\lambda(b)+t
	\end{equation}
	is an affine Lie bracket on $\mathfrak{g}$ if and only if the linear maps $\lambda,\mu\in\operatorname{Lin}(\mathfrak{g})$ satisfy that
	\begin{equation*}
		%\label{gen_der}
		[\lambda(a),b]-[a,\mu(b)]=\lambda([a,b]),
	\end{equation*}
	for all $a,b\in\mathfrak{g}$. What is more, every Lie affgebra arises in this way.
	
	\subsection{Extension of an algebra by a homothetic datum. Associative affgebras}\label{ssec.ass}
	In this section we recall the affinization of associative algebras from \cite{AndBrzRyb:ext}.
	\begin{definition}\label{def.ass.aff}
		An \textbf{associative affgebra} is an affgebra $(\mathfrak{a},\np)$ whose affine product $\np$ is associative, i.e., $(a\np b)\np c=a\np(b\np c)$ for all $a,b,c\in\mathfrak{a}$.
		
		An associative affgebra $\mathfrak{a}$ is said to be \textbf{quasi-commutative} if, for all $a,b,c\in \mathfrak{a}$,
		\begin{equation}\label{quasi-com}
			\langle a\np c,a\np b,c\np b\rangle = \langle b\np c,b\np a,c\np a\rangle.  
		\end{equation}
	\end{definition}
	Associative affgebras are a good affinization of associative algebras. This is justified by the following result, whose proof is similar to the one given in \cite[Theorem~4.3]{AndBrzRyb:ext} in the context of rings. 
	\begin{proposition}\label{prop.tan.ass}
		For any element $o$ of an associative affgebra $(\mathfrak{a},\np)$, the tangent space $T_o\mathfrak{a}$ is an associative algebra with product $\cdot_o$ (denoted by juxtaposition)
		\begin{gather}\label{ass.prod.tang}
			ab = a\np b - a\np o +o\np o - o\np b=\langle a\np b, a\np o, o\np o, o\np b,o \rangle,
		\end{gather}
		for all $a,b\in \mathfrak{a}$. Furthermore, for all $e\in \mathfrak{a}$, $(T_o\mathfrak{a},\cdot_o) \cong  (T_e\mathfrak{a},\cdot_e)$ via the translation map.
		
		If, in addition, $\mathfrak{a}$ is quasi-commutative, then $(T_o\mathfrak{a},\cdot_o)$ is commutative.
	\end{proposition}
	\begin{proof}
		The first two statements are proven in \cite[Theorem~4.3]{AndBrzRyb:ext}. Written in terms of the heap operation, the commutativity of the product $\cdot_o$ is equivalent to \eqref{quasi-com} with $c=o$. 
		%Since all the fibre algebras are mutually isomorphic \eqref{quasi-com} is satisfied for all $c=o$.
	\end{proof}
	\begin{definition}\label{Def: dp}
		Let $A$ be an associative algebra. A {\bf double operator} $\sigma=(\overleftarrow{\sigma},\overrightarrow{\sigma})$ on $A$ is a pair of linear endomorphisms:
		\begin{align*}
			\overleftarrow{\sigma}\colon A\rightarrow A, \qquad a \mapsto \overleftarrow{\sigma}(a)\coloneqq a\sigma,\\
			\overrightarrow{\sigma}\colon A\rightarrow A, \qquad a \mapsto \overrightarrow{\sigma}(a)\coloneqq \sigma a.
		\end{align*}
		
		The double operator $\sigma$ is called a \textbf{bimultiplication} \cite{Mac:ext} if, for all $a,b\in A$, the following conditions hold:
		\begin{equation}\label{bimultiplication}
			\sigma(ab)=(\sigma a)b,\quad (ab)\sigma=a(b\sigma), \quad a(\sigma b)=(a\sigma) b.
		\end{equation}
		
		If, additionally, $\sigma$ satisfies that 
		\begin{equation*}%\label{homothety}
			\sigma(a\sigma)=(\sigma a)\sigma
		\end{equation*}
		for all $a\in A$, then it is called a \textbf{double homothetism} \cite{Red:ver}.
	\end{definition}
	
	\begin{remark}\label{rem.deriv.bim}
		The associativity rules \eqref{bimultiplication} imply that if   $\sigma$ is a bimultiplication on an associative algebra $A$, then the map
		$$
		d^\sigma \colon A\to A, \qquad a\mapsto \sigma a-a \sigma,
		$$
		is a derivation on $A$.
	\end{remark}

	\begin{definition}\label{Def: hd}
		Let $A$ be an associative algebra. A pair $(\sigma,s)$, where $\sigma$ is a double homothetism and $s\in A$, is called a \textbf{homothetic datum} provided
		that
		\begin{gather}\label{commuts}
			\sigma s= s\sigma, \\  \label{homothetic datum} \sigma^2=\sigma+\bar{s},
		\end{gather}
		where $\bar{s}$ is an inner homothetism, that is, $\bar{s}$ is a double operator which multiplies by $s$ on the left and on the right sides.
	\end{definition}
	\begin{proposition}\label{affgebra_by_homothety}
		Let $A$ be an associative algebra. View $A$ as an affine space $(A,\langle-,-,-\rangle, -\triangleright_{-}-)$, where 
		\begin{equation*}
			\langle a,b,c\rangle\coloneqq a-b+c, \qquad \alpha\triangleright_{a}b \coloneqq (1-\alpha)a+\alpha b,
		\end{equation*}
		for all $\alpha\in\mathbb{F}$ and for all $a,b,c\in A$. Then, $A$ is an associative affgebra with the affine product given by
		\begin{equation*}%\label{aff.mult}
			a{\np} b\coloneqq ab+a\sigma+\sigma b+ s,
		\end{equation*}
		for all $a,b\in A$, where $(\sigma, s)$ is a homothetic datum on A. We denote this affgebra by $\aaa(\sigma, s)$. 
		
		Furthermore, if $A$ is a commutative algebra, then $\aaa(\sigma, s)$ is a quasi-commutative affgebra.
	\end{proposition}
	\begin{proof}
		It is proven in \cite[Theorem~3.6]{AndBrzRyb:ext}, that $\np$ makes $A$ into a truss, i.e.,\ an associative affgebra over $\mathbb{Z}$, so only the compatibility of the product $\np$ with the affine scalar multiplication needs to be shown. This follows by
		\begin{align*}
			(\alpha \triangleright_{x}y)\np b=&[(1-\alpha)x+\alpha y]b+[(1-\alpha)x+\alpha y]\sigma +\sigma b+s\\
			=&(1-\alpha)xb+\alpha yb+(1-\alpha)x\sigma+\alpha y\sigma +\sigma b+s\\
			=&(1-\alpha)xb+(1-\alpha)x\sigma+\sigma b+s-\alpha \sigma b-\alpha s+\alpha yb+\alpha y\sigma+\alpha \sigma b+\alpha s\\
			=&(1-\alpha)[xb+x\sigma+\sigma b+s]+\alpha[yb+y\sigma+\sigma b+s]\\
			=&(1-\alpha)(x\np b)+\alpha(y\np b)=\alpha\triangleright_{x\np b} (y\np b).
		\end{align*}
		Therefore, the product is affine in the first argument. By similar arguments we also obtain that $a\np -$ is affine. 
		
		The quasi-commutativity condition \eqref{quasi-com} for $\np$  follows from the commutativity of the product in $A$ by a straightforward calculation.
	\end{proof}
	
	The previous proposition provides a method for constructing associative affgebras from associative algebras using homothetic data. Besides, as it happens in the Lie setting \cite[Theorem 3.1]{BRZ3}, any associative affgebra is realized in this way, i.e., there is a 1-1 correspondence between associative affgebras and associative algebras supplemented with a homothetic datum. Indeed, given an associative affgebra $(\mathfrak{a},\np)$ and a point $o\in\mathfrak{a}$, the pair $(\sigma_o, s_o)$, where
	\begin{equation}\label{hom.o}
		a\sigma_o \coloneqq a\np o - o\np o,\quad \sigma_o a \coloneqq o\np a - o\np o, \quad s_o \coloneqq o\np o,
	\end{equation}
	is a homothetic datum on the associative algebra $(T_o\mathfrak{a},\cdot_{o})$, such that $(\mathfrak{a},\np) = \mathfrak{a}(\sigma_o,s_o)$. Moreover, if $\aaa(\sigma,s)$ is the affgebra constructed in Proposition~\ref{affgebra_by_homothety}, the fibre at any point $e\in A$, $T_e\mathfrak{a}(\sigma,s)$, coincides with $A$ up to isomorphism. Thus one can always view associative (and quasi-commutative) affgebras as arising from homothetic data on associative (and commutative) algebras. See concluding remarks in \cite[Section~4]{AndBrzRyb:ext}.
	
	\section{Affinization of derivations}\label{sec.der}
	The first new contribution of this work is the introduction of a natural notion of a derivation in the affine setting. In this section, we examine its main properties in order to justify the naturality of the concept.
	\begin{definition}\label{def.deriffation}
		Let $(\aaa,\np)$ be an associative affgebra and $D\colon \aaa \rightarrow \aaa$ an affine map. The \textbf{derivator} of $D$ is the bi-affine map
		\begin{equation*}
			\Delta^D\colon \aaa \times \aaa \rightarrow \aaa, \quad (a,b)\mapsto \langle D(a)\np b,D(a\np b),a\np D(b)\rangle.
		\end{equation*}
		
		We say that $D$ is an \textbf{affine derivation} on $\mathfrak{a}$ or, for short, a \textbf{deriffation} on $\mathfrak{a}$, if, for all $o\in \aaa$, the linearization of $\Delta^{D}$ at $o\in\mathfrak{a}$, which is the bilinear map $T_o\Delta^D\colon T_o\aaa\times T_o\aaa \rightarrow T_o\aaa$, defined by
		\[T_{o}\Delta^{D}(a,b)\coloneqq \Delta^{D}(a,b)-\Delta^{D}(a,o)+\Delta^{D}(o,o)-\Delta^{D}(o,b),\]
		is identically $o$, for all $a,b\in \aaa$.
		
	\end{definition}
	\begin{lemma}\label{independence_of_chosen_point}
		In the conditions of the previous definition, if $T_o\Delta^D=o$ for some $o\in \aaa$, then $T_o\Delta^D=o$ for all $o\in \aaa$.
	\end{lemma}
	\begin{proof}
		Let $o\in\aaa$ such that $T_o\Delta^D=o$. Thus, for all $a,b\in \aaa$,
		\begin{equation}\label{equaux}\Delta^D(a,b)=\Delta^{D}(a,o)-\Delta^{D}(o,o)+\Delta^{D}(o,b)=\langle \Delta^{D}(a,o),\Delta^{D}(o,o),\Delta^{D}(o,b)\rangle.\end{equation} 
		
		Using this formula, given a different $e\in\mathfrak{a}$, it is obtained that
		\begin{align*}
			\Delta^{D}(a,e)-\Delta^{D}(e,e)+\Delta^{D}(e,b)=&\langle \Delta^{D}(a,e),\Delta^{D}(e,e),\Delta^{D}(e,b)\rangle\\=&\langle\Delta^{D}(a,o),\Delta^{D}(o,o),\Delta^{D}(o,e),\\&\,\:\Delta^{D}(e,o),\Delta^{D}(o,o),\Delta^{D}(o,e),\\&\,\:\Delta^{D}(e,o),\Delta^{D}(o,o),\Delta^{D}(o,b)\rangle\;{\footnotesize\textnormal{(by \eqref{equaux})}}\\=&\langle\Delta^{D}(a,o),\Delta^{D}(o,o),\Delta^{D}(o,b)\rangle\;{\footnotesize\textnormal{(by \eqref{h_ab} and the Mal'cev identities)}}\\=&\Delta^{D}(a,b)\;{\footnotesize\textnormal{(by \eqref{equaux})}}.
		\end{align*}
		Therefore, $T_{e}\Delta^{D}=e$. 
	\end{proof}
	\begin{remark}
		Alternatively, the previous lemma can be proved as a consequence of the commutativity of the following diagram:
		
		\begin{minipage}{0.45\linewidth}
			\[\xymatrix{&T_{o}\aaa\times T_{o}\aaa\ar[d]_-{\tau_{o}^{e}\times\tau_{o}^{e}}\ar[rr]^-{T_{o}\Delta^{D}} & &T_{o}\aaa\ar[d]^-{\tau_{o}^{e}}\\&T_{e}\aaa\times T_{e}\aaa\ar[rr]^-{T_{e}\Delta^{D}} & &T_{e}\aaa.}\]
		\end{minipage}
		\begin{minipage}{0.59 \linewidth}
			\vspace{.3cm}
			Since $\tau_{o}^{e}$ is an isomorphism, it is obtained that
			\[T_{e}\Delta^{D}=\tau_{o}^{e}\circ T_{o}\Delta^{D}\circ ((\tau_{o}^{e})^{-1}\times(\tau_{o}^{e})^{-1})).\]
			
			Hence, if $T_{o}\Delta^{D}=o$, then
			\[T_{e}\Delta^{D}(a,b)=\tau_{o}^{e}(o)=\langle o,o,e\rangle=e.\]
		\end{minipage}
	\end{remark}
	
	In other words, Lemma \ref{independence_of_chosen_point} shows that the notion of a deriffation is independent of the chosen base point: if the derivator vanishes at some point $o\in\mathfrak{a}$, then it vanishes at any other point as well.
	\begin{proposition}\label{equiv.deriffation_and_derivation}
		An affine map $D\colon \aaa \rightarrow \aaa$ is a deriffation on an associative affgebra  $\aaa$ if and only if the linearization of $D$ at any $o\in \mathfrak{a}$, $T_oD$, is a derivation on $T_o\aaa$. %for all $o\in \aaa$.
	\end{proposition}
	\begin{proof}
		According to Lemma \ref{independence_of_chosen_point}, we only need to show the equivalence in a chosen point $o\in \aaa$. In this setting, it is obtained that
		\begin{align*}
			T_{o}D(ab)-T_{o}D(a)b-aT_{o}D(b)=&D(ab)-D(o)-D(a)b+D(o)b-aD(b)+aD(o)\\
			%=&D(a\np b-a\np o+o\np o-o\np b)-D(o)-D(a)b+D(o)b-aD(b)+aD(o)\\
			=&D(a\np b-a\np o+o\np o-o\np b)-D(o)\\&-D(a)\np b+D(a)\np o-o\np o+o\np b\\
			&+D(o)\np b-D(o)\np o+o\np o-o\np b\\&-a\np D(b)+a\np o-o\np o+o\np D(b)\\&+a\np D(o)-a\np o+o\np o-o\np D(o)\\
			=&D(a\np b)- D(a)\np b-a\np D(b)-D(a\np o)+ D(a)\np o+a\np D(o)\\
			&-D(o\np b)+ D(o)\np b+o\np D(b)+D(o\np o)- D(o)\np o-o\np D(o)\\
			=&-\Delta^D(a,b)+\Delta^D(a,o)-\Delta^D(o,o)+\Delta^D(o,b)\\
			=&-T_{o}\Delta^D(a,b).
		\end{align*}
		Therefore, the previous equality shows that $T_{o}D$ is a derivation on $T_{o}\aaa$ iff $T_{o}\Delta^{D}$ vanishes, which means that $D$ is a deriffation on $\mathfrak{a}$.
	\end{proof}
	\begin{corollary}\label{cor.deriv}
		For any associative algebra $A$ and any homothetic datum  $(\sigma,s)$, if $ d\colon A\rightarrow A$ is a linear endomorphism and $\delta \in A$, then the affine map $D\coloneqq d+\delta$ is a deriffation on $\aaa(\sigma,s)$ if and only if $d$ is a derivation on $A$.    
	\end{corollary}
	\begin{proof}
		This statement follows from Proposition~\ref{equiv.deriffation_and_derivation} and the facts that every affine map is a sum of a linear and a constant part, and that the tangent to the associative affgebra $\aaa(\sigma,s)$ obtained in Proposition \ref{affgebra_by_homothety} is isomorphic to $A$.
	\end{proof}
	
	\begin{definition}\label{def.inner}
		Let $(\mathfrak{a},\np)$ be an associative affgebra. A bi-affine map $\iota: \mathfrak{a}\times \mathfrak{a}\to \mathfrak{a}$ is called a \textbf{left} (respectively, \textbf{right}) \textbf{inner deriffation} if there exists an element $o\in \mathfrak{a}$ such that, for all $a\in \mathfrak{a}$, the linear map $T_o \iota(a,-)\colon T_o \mathfrak{a} \to T_o \mathfrak{a}$ (resp., $T_o \iota(-,a)$) is an inner derivation on $(T_o \mathfrak{a},\cdot_o)$. A left and right inner deriffation  $\iota$  such that
		$
		T_o\iota(a,b) = ab - ba$ in $(T_o\aaa, \cdot_{o})$, for all $a,b\in \aaa$,  is called an  \textbf{inner bi-deriffation}.
	\end{definition}
	
	\begin{lemma}\label{lem.inner}
		If $\iota\colon \mathfrak{a}\times \mathfrak{a} \to \mathfrak{a}$ is a left (resp., right) inner deriffation  on $(\mathfrak{a},\np)$, then, for all $e\in \mathfrak{a}$, $T_e \iota(a,-)\colon T_e \mathfrak{a} \to T_e \mathfrak{a}$ (resp., $T_e \iota(-,a)$) is an inner derivation on $(T_e \mathfrak{a},\cdot_e)$.
	\end{lemma}
	\begin{proof}
		If $\iota$ is a left inner deriffation, then there exist $o\in \mathfrak{a}$ and an affine map $f:\mathfrak{a}\to \mathfrak{a}$, such that, for all $a,b \in \mathfrak{a}$,
		$$
		\iota(a,b) = \langle f(a)\cdot _o b, b\cdot _of(a), \iota(a,o)\rangle
		$$
		in the algebra $(T_o\mathfrak{a},\cdot_{o})$. In particular, for any $e\in \mathfrak{a}$,
		$$
		\iota(a,e) = \langle f(a)\cdot_oe, e\cdot_of(a), \iota(a,o)\rangle,
		$$
		so that
		$$
		\begin{aligned}
			\iota(a,b) &= \langle f(a)\cdot _o b, f(a)\cdot_oe, e\cdot_of(a), b\cdot _of(a), \iota(a,e)\rangle\\
			&= \langle f(a)\cdot _o \tau_e^o(b), \tau_e^o(b)\cdot _of(a), \iota(a,e)\rangle \\
			&= \langle \tau_e^o \left(\tau_o^e (f(a))\cdot _e b\right), \tau_e^o\left(b\cdot _e\tau_o^e (f(a))\right), \iota(a,e)\rangle \\
			&= \langle \tau_o^e (f(a))\cdot _e b, b\cdot _e\tau_o^e (f(a)), \iota(a,e)\rangle ,
		\end{aligned}
		$$
		that is, $T_e\iota(a,-)$ is an inner derivation in $(T_e\mathfrak{a},\cdot_e)$, as required. The case of a right inner deriffation is dealt with by similar arguments.
	\end{proof}
	
	\begin{proposition}\label{prop.inner}
		Let $(\sigma, s)$ be a homothetic datum on an associative algebra $A$ and let $\aaa(\sigma,s)$ be the corresponding affgebra as in Proposition~\ref{affgebra_by_homothety}. Then, $\iota\colon A\times A\to A$ is an inner bi-deriffation on $\aaa(\sigma,s)$ if and only if there exist $l,m,t\in A$ such that, for all $a,b\in A$,
		\begin{equation}\label{inner.bi}
			\iota(a,b) = ab-ba +mb-bm +la-al +t. 
		\end{equation}
	\end{proposition}
	\begin{proof}
		In general,
		$$
		\iota(a,b) = ab-ba +\lambda(a) + \mu(b) +t,
		$$
		for some linear endomorphisms $\lambda,\mu$ of $A$ and $t\in A$. The map $\iota$ is a left inner deriffation if and only if $\iota(a,b) - \iota(a,0)$ is an inner derivation of $A$, that is, $\mu(b)$ is an inner derivation. Hence, necessarily, there exists $m\in A$ such that $\mu(b) = mb-bm$. Similarly, since $\iota$ is a right inner deriffation, there exists $l\in A$, such that $\lambda(a)=la-al$, as stated. That is, if $\iota$ is an inner bi-deriffation, then it is necessarily of the form \eqref{inner.bi}. The converse is obvious.
	\end{proof}
	
	\begin{lemma}[Commutator deriffations]\label{lem.com}
		For an associative affgebra $(\mathfrak{a}, \np)$ and a bi-affine map $f:\mathfrak{a}\times \mathfrak{a}\to \mathfrak{a}$, let  
		$$
		C^f : \mathfrak{a}\times \mathfrak{a}\to \mathfrak{a}, \qquad (a,b)\mapsto \langle a\np b , b\np a, f(a,b)\rangle. 
		$$
		Then, for all $a\in \mathfrak{a}$,  $C^f(a,-)$ (resp.,\ $C^f(-,a)$) is a deriffation if and only if $f(a,-)$ (resp.,\ $f(-,a)$) is a deriffation.
	\end{lemma}
	\begin{proof}
		One easily finds that, for all $a\in \mathfrak{a}$,
		$T_oC^f(a,-)$ comes out as
		$$
		b\mapsto a\np b-b\np a+o\np a-a\np o+f(a,b)-f(a,o)=(a\cdot_{o}b - b\cdot_{o}a) + (o\np b - b\np o) + (f(a,b) - f(a,o)).
		$$
		The first bracketed expression is an inner derivation on $T_o\mathfrak{a}$, and the second one is the derivation induced by a bimultiplication (see Remark~\ref{rem.deriv.bim} and the homothetic datum given in \eqref{hom.o}). Hence, the whole map is a derivation on $(T_o\mathfrak{a},\cdot_{o})$ if and only if the last bracketed expression, that is, the linearization of $f(a,-)$ in $T_o\mathfrak{a}$, is a derivation. 
		
		The second statement is proven by symmetric arguments.    
	\end{proof}

	\section{Poisson affgebras}\label{sec.Poisson}
	Relying on the notion of an affine derivation proposed in the previous section, this section is devoted to the introduction of Poisson affgebras. We present several approaches to the definition, as well as a result enabling the construction of such structures from a usual Poisson algebra. 
	\begin{definition}\label{def.aff.Poissona}
		Let $\aaa$ be an associative affgebra with affine product $\np$. We will say that $\aaa$ is a \textbf{one-sided Poisson affgebra} if there exists an affine Lie bracket $\{-,-\}$, called \textit{affine Poisson bracket}, on $\aaa$ such that, for all $z\in \aaa$, the affine maps $\{-,z\}$ are  deriffations on $\aaa$. If, additionally, $\{z,-\}$ is a deriffation for all $z\in \aaa$, then $\aaa$ is a \textbf{two-sided Poisson affgebra} or, simply, a \textbf{Poisson affgebra}.
		
		A \textbf{morphism of Poisson affgebras} is an affine map that preserves both the affine associative product and  the affine Poisson bracket, that is, $f\colon (\aaa,\np_{\aaa},\{-,-\}_{\aaa})\rightarrow (\mathfrak{b},\np_{\mathfrak{b}},\{-,-\}_{\mathfrak{b}})$ is a morphism of Poisson affgebras if
		\[f(a\np_{\aaa}b)=f(a)\np_{\mathfrak{b}} f(b),\quad f(\{a,b\}_{\mathfrak{a}})=\{f(a),f(b)\}_{\mathfrak{b}},\]
		for all $a,b\in\aaa$.
	\end{definition}
	
	Since a Poisson algebra is a combination of an associative algebra and a Lie algebra whose products are compatible via the Leibniz rule \eqref{Leib.rule}, an approach that may be followed to define Poisson affgebras is the general procedure of affinization of the Leibniz rule described in \cite{BRP}. In this sense, a Poisson affgebra would be a triple $(\aaa, \np,\{-,-\})$, where $(\aaa,\np)$ is an associative affgebra and $(\aaa,\{-,-\})$ is a Lie affgebra, such that 
	\begin{equation*}
		\langle \{x,z\}\np y ,\{x\np y,z\}, x\np\{y,z\}\rangle= \mathcal{P}(x,y,z)
	\end{equation*}
	for all $x,y,z\in \aaa$, where $\mathcal{P}\colon \aaa \times \aaa \times \aaa \rightarrow \aaa$ is a tri-affine map, called the \textit{Poissonian}, whose linearization at any point $o\in\aaa$ vanishes, i.e., $T_{o}\mathcal{P}=o$ for all $o\in\aaa$. However, one easily finds out that $\mathcal{P}(x,y,z)=\Delta^{\{-,z\}}(x,y)$ for all $x,y,z\in \aaa$, where $\Delta^{\{-,z\}}$ is the derivator of ${\{-,z\}}:\mathfrak{a}\to\aaa$ (see Definition~\ref{def.deriffation}). Consequently, 
	\[T_{o}\mathcal{P}=o\iff T_{o}\Delta^{\{-,z\}}=o\textnormal{ for all }z\in\aaa\]
	or, in other words, the vanishing of $T_{o}\mathcal{P}$ is equivalent to the fact that $\{-,z\}$ is a deriffation on $\aaa$ for all $z$. Hence, the approach based on the Poissonian is equivalent to the notion of one-sided Poisson affgebras.

	The next result shows why Poisson affgebras are the right affinization of Poisson algebras.
	\begin{proposition}
		If $(\aaa, \np,\{-,-\})$ is a one-sided Poisson affgebra, then $(T_{o}\aaa, \cdot_{o},[-,-])$ is a Poisson algebra for all $o\in\aaa$, where $\cdot_{o}$ is defined as in \eqref{ass.prod.tang}, and $[-,-]$ as in \eqref{linearLiebr}.
	\end{proposition}
	\begin{proof}
		Let $o\in\aaa$. Taking into account \cite[Theorem 2.6]{BRZ3}, which shows that $(T_{o}\aaa,[-,-])$ is a Lie algebra, and Proposition~\ref{prop.tan.ass}, it only remains to prove that \eqref{Leib.rule} holds. Indeed, a routine calculation yields the following:
		
		\begin{align}\label{tangentpoisson}
			\begin{split}
				[x\cdot_{o}y,z]-[x,z]\cdot_{o}y-x\cdot_{o}[y,z]=&-\{x,z\}\np y+\{x\np y,z\}-x\np\{y,z\}\\&+\{x,z\}\np o-\{x\np o,z\}+x\np\{o,z\}\\&-\{o,z\}\np o+\{o\np o,z\}-o\np\{o,z\}\\&+\{o,z\}\np y-\{o\np y,z\}+o\np\{y,z\}\\&+\{x,o\}\np y-\{x\np y,o\}+x\np\{y,o\}\\&-\{x,o\}\np o+\{x\np o,o\}-x\np\{o,o\}\\&+\{o,o\}\np o-\{o\np o,o\}+o\np\{o,o\}\\&-\{o,o\}\np y+\{o\np y,o\}-o\np\{y,o\}.  
			\end{split}
		\end{align}
		
		Since $\{-,z\}$ is a deriffation for all $z\in\aaa$, the linearization of the derivator $\Delta^{\{-,z\}}$ at any $o\in\aaa$ vanishes, that is,
		\begin{align*}
			o=T_{o}\Delta^{\{-,z\}}(x,y)=&\Delta^{\{-,z\}}(x,y)-\Delta^{\{-,z\}}(x,o)+\Delta^{\{-,z\}}(o,o)-\Delta^{\{-,z\}}(o,y)\\=&\{x,z\}\np y-\{x\np y,z\}+x\np\{y,z\}\\&-\{x,z\}\np o+\{x\np o,z\}-x\np\{o,z\}\\&+\{o,z\}\np o-\{o\np o,z\}+o\np\{o,z\}\\&-\{o,z\}\np y+\{o\np y,z\}-o\np\{y,z\}
		\end{align*}
		for all $x,y,z\in\aaa$. From the above formula, we deduce that the block formed by the first four lines of \eqref{tangentpoisson} coincides with $-T_{o}\Delta^{\{-,z\}}(x,y)$, while the block formed by the last four lines corresponds to $T_{o}\Delta^{\{-,o\}}(x,y)$. Hence,
		\[[x\cdot_{o}y,z]-[x,z]\cdot_{o}y-x\cdot_{o}[y,z]=\underbrace{-T_{o}\Delta^{\{-,z\}}(x,y)}_{=o}+\underbrace{T_{o}\Delta^{\{-,o\}}(x,y)}_{=o}=o.\qedhere\]
	\end{proof}

	The next theorem provides one with the  Poisson affgebra version of \cite[Theorem~3.1]{BRZ3} in the Lie setting or \cite[Theorem~3.5]{BRP} for Leibniz affgebras. In essence, it proves that any Poisson affgebra is obtained from a Poisson algebra equipped with a pair of derivations and a homothetic datum.
	\begin{theorem}\label{main.th.poisson}
		Let $(P,\cdot ,[-,-])$ be a Poisson algebra. Let  $(\sigma,s)$ be a homothetic datum on the associative algebra $(P,\cdot)$ and $\lambda, \mu\colon P\rightarrow P$ linear endomorphisms satisfying that $(\lambda,-\mu,\lambda)$ is a generalized derivation on $(P,[-,-])$, i.e., the equality
		\begin{equation}\label{gen_der}
			\lambda([a,b])=[\lambda(a),b]-[a,\mu(b)]
		\end{equation}
		holds. Then, the associative affgebra $\paff(\sigma,s)$ induced by the homothetic datum $(\sigma,s)$ as in Proposition~\ref{affgebra_by_homothety} is a Poisson affgebra with the affine Poisson bracket
		\begin{equation} 
			\{a,b\}=[a,b]+\mu(a)+\lambda(b)+t,
		\end{equation}
		with $t\in P$, if and only if $\lambda,\mu$ are derivations on $(P,\cdot)$. We denote the resulting Poisson affgebra by $\paff(P;\sigma, s;\mu,\lambda,t).$ Furthermore, for all $o\in P$,
		\begin{gather}\label{isotan}T_{o}\paff(P;\sigma, s;\mu,\lambda,t)\cong P.\end{gather}
		
		Conversely, given any Poisson affgebra $(\aaa,\np,\{-,-\})$ and a point $o\in\aaa$, there exist linear endomorphisms $\mu,\lambda\colon T_{o}\aaa\rightarrow T_{o}\aaa$ (which are necessarily derivations on $(T_{o}\aaa,\cdot_{o})$ satisfying \eqref{gen_der}) and a point $t\in \mathfrak{a}$ such that $\aaa=\mathfrak{p}(T_{o}\aaa;\sigma_{o},s_{o};\mu,\lambda,t)$, where $(\sigma_{o},s_{o})$ is the homothetic datum given in \eqref{hom.o}.
	\end{theorem}
	\begin{proof}
		According to Proposition~\ref{affgebra_by_homothety}, $\paff(\sigma,s)$ is an associative affgebra, and \cite[Theorem~3.1]{BRZ3} states that \eqref{affine Lie bracket} defines an affine Lie bracket on $\paff(\sigma,s)$ if and only if \eqref{gen_der} holds. Due to these facts, for $\paff(\sigma,s)$ to be a Poisson affgebra, we only need to show that the maps $\{-,z\},\{z,-\}\colon P\rightarrow P$ are deriffations on the associative affgebra $\paff(\sigma,s)$ for all $z$. By Proposition~\ref{equiv.deriffation_and_derivation}, let us prove that $T_{0}\{-,z\}$ and $T_{0}\{z,-\}$ are derivations on $(P,\cdot)=\Tan{0}{\paff(\sigma,s)}$. Indeed, for all $a,b\in P$,
		\begin{align*}
			&\Tan{0}{\{-,z\}}(ab)=\Tan{0}{\{-,z\}}(a)b+a\Tan{0}{\{-,z\}}(b)\\
			\iff&\{ab,z\}-\{0,z\}=(\{a,z\}-\{0,z\})b+a(\{b,z\}-\{0,z\})\\
			\iff&[ab,z]+\mu(ab)=([a,z]+\mu(a))b+a([b,z]+\mu(b))\\
			\iff&\mu(ab)=\mu(a)b+a\mu(b)\;{\footnotesize\textnormal{(by \eqref{Leib.rule} for $P$)}}
		\end{align*}
		and, similarly,
		\begin{align*}
			\Tan{0}{\{z,-\}}(ab)=\Tan{0}{\{z,-\}}(a)b+a\Tan{0}{\{z,-\}}(b)\iff \lambda(ab)=\lambda(a)b+a\lambda(b),
		\end{align*}
		i.e., $\{-,z\},\{z,-\}$ are deriffations on $\paff(\sigma,s)$ if and only if $\mu,\lambda$ are derivations on $P$.
		
		To prove the isomorphism given in \eqref{isotan}, note that the Lie bracket $[-,-]'$ in $T_{o}\paff(P;\sigma, s;\mu,\lambda,t)$ is given by 
		\[[a,b]'=[a,b]-[a,o]+[o,o]-[o,b]+o,\]
		whereas the associative product $\star$ is defined as follows:
		\[a\star b=ab-ao+o^{2}-ob+o.\]
		Therefore, the isomorphism we are looking for is $\varphi(a)=a-o$ for all $a\in P$.
		
		To conclude, given a Poisson affgebra $(\aaa,\np,\{-,-\})$ and a point $o\in\aaa$, it is enough to consider 
		\[\mu(a)=\{a,o\}-\{o,o\},\quad\lambda(a)=\{o,a\}-\{o,o\},\quad t=\{o,o\},\]
		and then proceed as in \cite[Theorem~3.5]{BRP}.
	\end{proof}
	\begin{remark}
		If one merely asks for $\paff(P;\sigma, s;\mu,\lambda,t)$ to be a one-sided Poisson affgebra, then it suffices that $\mu$ is a derivation on $P$.
	\end{remark}
	
	As the first application of Theorem~\ref{main.th.poisson}, we determine all Poisson affgebras with commutator Poisson bracket on any, equivalently every, fibre, and establish that any associative affgebra admits a Poisson affgebra structure. We begin by introducing a special type of maps on an associative affgebra.
	
	\begin{definition}\label{def.ubicentral}
		Let $(\aaa, \np)$ be an associative affgebra. An affine map $f\colon\aaa \to \aaa$ is said to be \textbf{ubicentral} provided that, for all $a,b,c\in \aaa$,
		\begin{equation}\label{eq.ubi}
			\langle f(a)\np b, f(a)\np c, b\np f(c), c\np f(c), f(c)\np c\rangle = \langle b \np f(a), c\np f(a), f(c)\np b\rangle.
		\end{equation}
	\end{definition}
	The set of ubicentral maps is never empty as every constant map is ubicentral.
	
	\begin{lemma}\label{lem.ubicentral}
		Let $(\aaa, \np)$ be an associative affgebra and $f\colon\aaa\rightarrow\aaa$ an affine map. The following statements are equivalent:
		\begin{itemize}
			\item[(i)] $f\colon\aaa\to\aaa$ is ubicentral in $(\aaa,\np)$.
			\item[(ii)] The linearisation of $f$ in all algebra fibres of $(\aaa,\np)$ is central. 
			\item[(iii)] There is an algebra fibre $(T_o\aaa, \cdot_{o})$ in which $T_of$ is central.
		\end{itemize}
	\end{lemma}
	\begin{proof}
		The linearisation $T_cf$ of $f$ in $T_c\aaa$ is $a\mapsto f(a) - f(c)$, and it is central if and only if, for all $a,b\in T_c\aaa$,
		\begin{equation}\label{central}
			(f(a)-f(c))\cdot_c b = b \cdot_c (f(a)-f(c)).
		\end{equation}
		Written in terms of the product $\np$, this is exactly \eqref{eq.ubi}, proving the first equivalence.
		
		To prove the second equivalence, assume that the linear map $f - f(o)$ is central in the algebra $(T_o\aaa, \cdot_o)$. Then, for all $a,b,c\in \aaa$,
		$$
		\begin{aligned}
			(f(a)-f(c))\cdot_o b &= (f(a)-f(o))\cdot_o b - (f(c)-f(o))\cdot_o b\\
			&= b\cdot_o (f(a)-f(o)) - b\cdot_o (f(c)-f(o))\\
			&= b \cdot_o (f(a)-f(c)).
		\end{aligned}
		$$
		For the convenience of the reader, it is worth reminding that $f(a)-f(c)$ in the computation above is a subtraction in $T_{o}\aaa$.
		Now, applying the translation algebra isomorphism $\tau_o^c$ to the equality $(f(a)-f(c))\cdot_o \tau_c^o(b) = \tau_c^o (b) \cdot_o (f(a)-f(c))$ one obtains the equality \eqref{central} in $(T_c\aaa, \cdot_c)$ as required.
	\end{proof}
	
	Note in passing that by Lemma~\ref{lem.ubicentral} it is sufficient for the condition \eqref{eq.ubi} to hold for any $c$ in order to hold for all $c$.
	\begin{proposition}[Commutator Poisson affgebras]\label{prop.comm}
		Let $(\paff,\np)$ be an associative affgebra. Let $D$  be a deriffation on $(\paff,\np)$ and let $F$ be a ubicentral deriffation on $(\paff,\np)$. Then, $(\paff,\np)$ is a Poisson affgebra with the bracket 
		\begin{equation}\label{comm}
			\{a,b\} = \langle a\np b , b\np a, D(a),D(b),F(b)\rangle,
		\end{equation}
		for all $a,b\in \paff$. Furthermore, every Poisson algebra fibre of $(\paff, \np, \{-,-\})$ has the Poisson bracket  given by the commutator. Any Poisson affgebra with commutator Poisson algebra fibres is obtained in this way.
	\end{proposition}
	\begin{proof}
		Take any $o\in \paff$. Viewing $\paff$ as $\mathfrak{p}(T_{o}\paff;\sigma_{o},s_{o};\mu,\lambda,t)$ one easily finds that 
		\begin{equation}\label{comm.bracket}
			\{a,b\} = ab-ba +\mu(a) +  (\kappa-\mu)(b) + t,
		\end{equation}
		\begin{equation}\label{lam.d}
			\mu(a) \coloneqq D(a) - d^{\sigma_o}(a) -D(o), \qquad \kappa(a) \coloneqq F(a) - F(o), \qquad t \coloneqq F(o),
		\end{equation}
		where $d^{\sigma_o}$ is the derivation on $(T_o\paff,\cdot)$ defined in Remark~\ref{rem.deriv.bim}. Since $D$ and $F$ are deriffations on $\paff$, $\mu$ and $\lambda =  \kappa - \mu$ are derivations on $T_o\paff$ by Corollary~\ref{cor.deriv}, and consequently derivations of the commutator bracket $[a,b]=ab-ba$.  The condition \eqref{gen_der} is equivalent to $\kappa$ being a central map in $T_o\paff$, and hence to $F$ being a ubicentral map by Lemma~\ref{lem.ubicentral}. Indeed,
		\begin{align*}
			\eqref{gen_der}\iff& \kappa([a,b])-\mu([a,b])=[\kappa(a),b]-[\mu(a),b]-[a,\mu(b)]\;{\footnotesize\textnormal{(by $\lambda=\kappa-\mu$)}}\\\iff&\kappa([a,b])=[\kappa(a),b] \;{\footnotesize\textnormal{(since $\mu$ is a derivation)}}\\\iff& [\kappa(a),b]+[a,\kappa(b)]=[\kappa(a),b] \;{\footnotesize\textnormal{(since $\kappa$ is a derivation)}}\\\iff&[a,\kappa(b)]=0\iff \kappa\textnormal{ is a central map.} 
		\end{align*}
		Therefore, by Theorem~\ref{main.th.poisson}, $\{-,-\}$ is an affine Poisson bracket with commutator linearisations. 
		
		In the converse direction, one uses correspondence \eqref{lam.d} to relate back the bracket \eqref{comm.bracket} to the bracket \eqref{comm}.
	\end{proof}
	\begin{remark}\label{rem.one-sided.comm}
		If one wants merely to study one-sided commutator Poisson affgebras, then the requirement that $F$ is a ubicentral deriffation should be replaced by 
		\begin{align}\label{cond_assoc_Pois_2}
			\begin{split}
				\langle F(a\np b), &F(b\np a), b\np F(a), F(a)\np b,F(b\np c),F(c\np b),F(a)\np c\rangle \\
				&=\langle c\np F(a),F(c\np a),F(a\np c),F(c)\np b,b\np F(c), c\np F(c),F(c)\np c\rangle,
			\end{split}
		\end{align}
		for all $a,b,c\in\aaa$, which, in terms of the map $\kappa\coloneqq T_{o}F$ defined above, means that $\kappa([a,b])=[\kappa(a),b]$ for all $a,b\in T_{o}\paff$, i.e., $\kappa$ is in the \textbf{centroid} (see \cite{LL}) of $T_{o}\paff$ with the commutator bracket.
	\end{remark}
	\begin{corollary}\label{cor.comm}
		Any associative affgebra $\aaa$ admits an affine Poisson bracket.
	\end{corollary}
	\begin{proof}
		Since any constant map $\aaa\to\aaa$ is a ubicentral deriffation, by setting $D=\mathrm{const.}$ and $F=t$ with $t\in \aaa$ in \eqref{comm}, one obtains a commutator Poisson bracket $\{a,b\} = \langle a \np b, b\np a, t\rangle$.
	\end{proof}
	
	\begin{example}[Darboux's Poisson affgebra]
		The preceding result allows us to construct an example of an infinite dimensional Poisson affgebra whose fibre is the Poisson algebra $\mathbb{R}[x,y]$, the ring of polynomials in two variables, with the Darboux bracket \eqref{darboux}: \[[f,g]=\frac{\partial f}{\partial x}\frac{\partial g}{\partial y}-\frac{\partial f}{\partial y}\frac{\partial g}{\partial x},\] which is a Poisson subalgebra of $C^{\infty}(\mathbb{R}^{2})$. In this example we work with only two variables to make the computations as simple as possible. 
		
		Thus, the problem reduces to finding derivations $\mu,\lambda$ in $\mathbb{R}[x,y]$ satisfying \eqref{gen_der}. Recall that derivations on $C^{\infty}(\mathbb{R}^{2})\supset \mathbb{R}[x,y]$ are exactly the smooth vector fields on $\mathbb{R}^{2}$. Hence, $\mu,\lambda$ must be of the form 
		\[\mu=\mu_{1}\frac{\partial }{\partial x}+\mu_{2}\frac{\partial }{\partial y},\qquad\lambda=\lambda_{1}\frac{\partial }{\partial x}+\lambda_{2}\frac{\partial }{\partial y},\]
		with $\mu_{i},\lambda_{i}\in \mathbb{R}[x,y]$ for $i=1,2$. Therefore, to derive the conditions on $\mu$ and $\lambda$ ensuring that \eqref{gen_der} holds, suffices it to evaluate the identity \eqref{gen_der} on monomials, $M\coloneqq\{f_{m,n}(x,y)\coloneqq x^{m}y^{n}\,\colon\,m,n\in\mathbb{N}\}$, as $\mu$ and $\lambda$ are linear and the monomials generate $\mathbb{R}[x,y]$ as a real vector space. Let $f_{m,n},f_{k,l}\in M$. For the computations that follow, we shall make use of the following properties of monomials:
		\begin{gather*}
			\frac{\partial f_{m,n}}{\partial x}=mf_{m-1,n},\qquad \frac{\partial f_{m,n}}{\partial y}=nf_{m,n-1},\qquad f_{m,n}f_{k,l}=f_{m+k,n+l},\\
			[f_{m,n},f_{k,l}]=(ml-nk)f_{m+k-1,n+l-1},\\
			\mu(f_{m,n})=m\mu_{1}f_{m-1,n}+n\mu_{2}f_{m,n-1},\qquad \lambda(f_{m,n})=m\lambda_{1}f_{m-1,n}+n\lambda_{2}f_{m,n-1}.
		\end{gather*}
		On the one hand, 
		\begin{align}\label{genderDarb1}
			\begin{split}
				\lambda([f_{m,n},f_{k,l}])=&(ml-nk)\lambda(f_{m+k-1,n+l-1})\\=&(ml-nk)\big[(m+k-1)\lambda_{1}f_{m+k-2,n+l-1}+(n+l-1)\lambda_{2}f_{m+k-1,n+l-2}\big].
			\end{split}
		\end{align}
		On the other hand,
		$$
		%\resizebox{\linewidth}{!}{
			\begin{aligned}
				[\lambda(f_{m,n}),&f_{k,l}]=[m\lambda_{1}f_{m-1,n}+n\lambda_{2}f_{m,n-1},f_{k,l}]\\=&m[\lambda_{1}f_{m-1,n},f_{k,l}]+n[\lambda_{2}f_{m,n-1},f_{k,l}]
				\\=&m\left(\frac{\partial \lambda_{1}f_{m-1,n}}{\partial x}\frac{\partial f_{k,l}}{\partial y}-\frac{\partial \lambda_{1}f_{m-1,n}}{\partial y}\frac{\partial f_{k,l}}{\partial x}\right)+n\left(\frac{\partial \lambda_{2}f_{m,n-1}}{\partial x}\frac{\partial f_{k,l}}{\partial y}-\frac{\partial \lambda_{2}f_{m,n-1}}{\partial y}\frac{\partial f_{k,l}}{\partial x}\right)\\=&m\left\{l\left(\frac{\partial\lambda_{1}}{\partial x}f_{m-1,n}+(m-1)\lambda_{1}f_{m-2,n}\right)f_{k,l-1}-k\left(\frac{\partial \lambda_{1}}{\partial y}f_{m-1,n}+n\lambda_{1}f_{m-1,n-1}\right)f_{k-1,l}\right\}\\&+n\left\{l\left(\frac{\partial \lambda_{2}}{\partial x}f_{m,n-1}+m\lambda_{2}f_{m-1,n-1}\right)f_{k,l-1}-k\left(\frac{\partial \lambda_{2}}{\partial y}f_{m,n-1}+(n-1)\lambda_{2}f_{m,n-2}\right)f_{k-1,l}\right\}\\=&\left(ml\frac{\partial \lambda_{1}}{\partial x}-nk\frac{\partial \lambda_{2}}{\partial y}\right)f_{m+k-1,n+l-1}+\big[ml(m-1)\lambda_{1}-kmn\lambda_{1}\big]f_{m+k-2,n+l-1}\\&-km\frac{\partial\lambda_{1}}{\partial y}f_{m+k-2,n+l}+nl\frac{\partial \lambda_{2}}{\partial x}f_{m+k,n+l-2}+\big[mnl\lambda_{2}-kn(n-1)\lambda_{2}\big]f_{m+k-1,n+l-2},
			\end{aligned}
			%}
		$$
		and then, using the previous identity, one obtains that
		\begin{align}\label{genderDarb2}
			\begin{split}
				[\lambda(f_{m,n}),f_{k,l}]-&[f_{m,n},\mu(f_{k,l})]=[\lambda(f_{m,n}),f_{k,l}]+[\mu(f_{k,l}),f_{m,n}]\\=&\left\{ml\left(\frac{\partial\lambda_{1}}{\partial x}-\frac{\partial \mu_{2}}{\partial y}\right)+nk\left(\frac{\partial\mu_{1}}{\partial x}-\frac{\partial \lambda_{2}}{\partial y}\right)\right\}f_{m+k-1,n+l-1}\\&+\left\{\big[ml(m-1)-kmn\big]\lambda_{1}+\big[kn(k-1)-kml\big]\mu_{1}\right\}f_{m+k-2,n+l-1}
				\\&-\left\{km\left(\frac{\partial \lambda_{1}}{\partial y}+\frac{\partial \mu_{1}}{\partial y}\right)\right\}f_{m+k-2,n+l}\\&+\left\{nl\left(\frac{\partial \lambda_{2}}{\partial x}+\frac{\partial \mu_{2}}{\partial x}\right)\right\}f_{m+k,n+l-2}\\&+\left\{\big[mln-kn(n-1)\big]\lambda_{2}+\big[nkl-ml(l-1)\big]\mu_{2}\right\}f_{m+k-1,n+l-2}.
			\end{split}
		\end{align}
		Consequently, the required conditions are obtained by comparing the coefficients of monomials of equal degree in equalities \eqref{genderDarb1} and \eqref{genderDarb2}:
		\begin{itemize}
			\item $f_{m+k-2,n+l-1}$
			\begin{align*}
				&(ml-nk)(m+k-1)\lambda_{1}=\big[ml(m-1)-kmn\big]\lambda_{1}+\big[kn(k-1)-kml\big]\mu_{1}\\\iff&[kml+nk-nk^{2}](\lambda_{1}+\mu_{1})=0\iff \lambda_{1}+\mu_{1}=0.
			\end{align*}
			\item $f_{m+k-1,n+l-2}$
			\begin{align*}
				&(ml-nk)(n+l-1)\lambda_{2}=\big[mln-kn(n-1)\big]\lambda_{2}+\big[nkl-ml(l-1)\big]\mu_{2}\\\iff& [ml^{2}-ml-nkl](\lambda_{2}+\mu_{2})=0\iff\lambda_{2}+\mu_{2}=0.
			\end{align*}
		\end{itemize}
		Therefore, $\mu=-\lambda$ and then, the conditions obtained by matching the coefficients of the monomials $f_{m+k-2,n+l}$ and $f_{m+k,n+l-2}$ are automatically satisfied. Thus, it only remains to determine the conditions arising from the monomial $f_{m+k-1,n+l-1}$:
		\begin{itemize}
			\item $f_{m+k-1,n+l-1}$
			\begin{align*}
				&ml\left(\frac{\partial\lambda_{1}}{\partial x}-\frac{\partial \mu_{2}}{\partial y}\right)+nk\left(\frac{\partial\mu_{1}}{\partial x}-\frac{\partial \lambda_{2}}{\partial y}\right)=0\\\overset{\mu=-\lambda}{\iff}&(ml-nk)\left(\frac{\partial\lambda_{1}}{\partial x}+\frac{\partial \lambda_{2}}{\partial y}\right)=0\iff \frac{\partial\lambda_{1}}{\partial x}+\frac{\partial \lambda_{2}}{\partial y}=0.
			\end{align*}
		\end{itemize}
		
		So, we conclude that $\mathbb{R}[x,y]$ is a Poisson affgebra with the affine Poisson bracket given by 
		\[\{f,g\}=[f,g]+\lambda(g-f)+h,\]
		where $h$ is any polynomial in $\mathbb{R}[x,y]$ and $\lambda=\lambda_{1}\frac{\partial}{\partial x}+\lambda_{2}\frac{\partial}{\partial y}$ such that, if $\lambda_{1}$ is a polynomial of the form $\lambda_{1}(x,y)=\sum_{k,l}\alpha_{kl}x^{k}y^{l}$, $\alpha_{kl}\in\mathbb{R}$, then:
		\begin{align*}
			& \frac{\partial\lambda_{2}}{\partial y}(x,y)=-\frac{\partial\lambda_{1}}{\partial x}(x,y)=-\sum_{k,l}k\alpha_{kl}x^{k-1}y^{l}\\\implies& \lambda_{2}(x,y)=-\sum_{k,l}k\alpha_{kl}x^{k-1}\int y^{l}\,{\rm d}y=-\sum_{k,l}k\alpha_{kl}x^{k-1}\left(\frac{y^{l+1}}{l+1}+c(x)\right).
		\end{align*}
	\end{example}
	\begin{theorem}\label{hom.Poisson.aff}
		Let $(P,\cdot,[-,-])$ and $(P',\ast,[-,-]')$ be Poisson algebras. An affine map $$\varphi \colon \paff(P;\sigma,s;\mu,\lambda,t) \rightarrow \paff(P';\sigma',s';\mu',\lambda',t')$$ is a morphism of Poisson affgebras if and only if there exist a Poisson algebra morphism $\psi\colon P\rightarrow P'$ and a point $q'\in P'$ such that, for all $a\in P$, the following conditions hold:
		\begin{subequations}\label{hom.a-f}
			\begin{gather}\label{hom_a}
				\psi(s)+q'=(q')^{2}+q'\sigma'+\sigma'q'+s',\\
				\label{hom_b}
				\psi(\sigma a)=q'\ast \psi(a)+\sigma'\psi(a),\\
				\label{hom_c}
				\psi(a\sigma)=\psi(a)\ast q'+\psi(a)\sigma',\\
				\label{hom_d}
				\psi(t)+q'=(\mu'+\lambda')(q')+t',\\
				\label{hom_e}
				\psi(\mu(a))=(\mu'+{\rm ad}_{q'}')(\psi(a)),\\
				\label{hom_f}
				\psi(\lambda(a))=(\lambda'-{\rm ad}'_{q'})(\psi(a)).
			\end{gather}
		\end{subequations}
	\end{theorem}
	\begin{proof}
		If $\varphi\colon \paff(P;\sigma,s;\mu,\lambda,t) \rightarrow \paff(P';\sigma',s';\mu',\lambda',t')$ is an affine map, then there exists $\psi\colon P\rightarrow P'$ a linear transformation such that $\psi(a)=\varphi(a)-\varphi(0)$. Setting $q'\coloneqq \varphi(0)$, we obtain that $\varphi=\psi+q'$.
		
		We will derive first the conditions for $\varphi$ to be a morphism of associative affgebras, that is,
		\begin{align}\label{cond.1homomorphism}
			\begin{split}
				&\varphi(a\np b)=\varphi(a)\np' \varphi(b)\\
				\iff&\psi(a\np b)+q'=(\psi(a)+q')\np'(\psi(b)+q')\\
				\iff&\psi(ab)+\psi(a\sigma)+\psi(\sigma b)+\psi(s)+q'\\
				&=\psi(a)\ast \psi(b)+\psi(a)\ast q'+q'\ast \psi(b)+q'\ast q'+\psi(a)\sigma'+q'\sigma'+\sigma'\psi(b)+\sigma'q'+s'.
			\end{split}
		\end{align}
		Setting $a=b=0$ in \eqref{cond.1homomorphism} one obtains \eqref{hom_a}, and thus \eqref{cond.1homomorphism} reduces to
		\begin{equation}\label{cond.2homomorphism}
			\psi(ab)+\psi(a\sigma)+\psi(\sigma b)=\psi(a)\ast \psi(b)+\psi(a)\ast q'+q'\ast \psi(b)+\psi(a)\sigma'+\sigma'\psi(b).
		\end{equation}
		Setting $a=0$ in \eqref{cond.2homomorphism} we obtain \eqref{hom_b}, and then \eqref{cond.2homomorphism} becomes
		\begin{equation}\label{cond.3homomorphism}
			\psi(ab)+\psi(a\sigma)=\psi(a)\ast \psi(b)+\psi(a)\ast q'+\psi(a)\sigma'.
		\end{equation}
		Finally, setting $b=0$ in \eqref{cond.3homomorphism} yields \eqref{hom_c} as well as the fact that $\psi$ preserves the associative product between the algebras $(P,\cdot)$ and $(P',\ast)$. 
		
		In order for $\varphi$ to preserve the affine Poisson bracket, the following conditions must be also satisfied:
		\begin{align}\label{cond.4homomorphism}
			\begin{split}
				&\varphi(\{a,b\})=\{\varphi(a),\varphi(b)\}'\\
				\iff&\psi(\{a,b\})+q'=\{\psi(a)+q',\psi(b)+q'\}'\\
				\iff&\psi([a,b])+\psi(\mu(a))+\psi(\lambda(b))+\psi(t)+q'\\
				&=[\psi(a),\psi(b)]'+[\psi(a),q']'+[q',\psi(b)]'+\mu'(\psi(a))+\mu'(q')+\lambda'(\psi(b))+\lambda'(q')+t'.
			\end{split}
		\end{align}
		Setting $a=b=0$ in \eqref{cond.4homomorphism} leads to \eqref{hom_d}, and thus \eqref{cond.4homomorphism} simplifies to
		\begin{equation}\label{cond.5homomorphism}
			\psi([a,b])+\psi(\mu(a))+\psi(\lambda(b))=[\psi(a),\psi(b)]'+[\psi(a),q']'+[q',\psi(b)]'+\mu'(\psi(a))+\lambda'(\psi(b)).
		\end{equation}
		Setting $b=0$ in \eqref{cond.5homomorphism}, \eqref{hom_e} is obtained, and then \eqref{cond.5homomorphism} reduces to
		\begin{equation}\label{cond.6homomorphism}
			\psi([a,b])+\psi(\lambda(b))=[\psi(a),\psi(b)]'+[q',\psi(b)]'+\lambda'(\psi(b)).
		\end{equation}
		Finally, setting $a=0$ in \eqref{cond.6homomorphism} yields \eqref{hom_f} as well as the fact that $\psi$ is a morphism of Lie algebras between $(P,[-,-])$ and $(P',[-,-]')$.
		
		The converse is straightforward.
	\end{proof}
	
	The next corollary provides a characterization of the isomorphisms between the type of Poisson affgebras introduced in Theorem~\ref{main.th.poisson}. In contrast with the preceding case, the fact that $\varphi$ is bijective ensures that the homothetic datum $(\sigma',s')$, the derivations $\mu'$ and $\lambda'$, as well as the element $t'$, are uniquely determined. This result enables the classification of Poisson affgebras up to isomorphisms in terms of their dimension, as we shall see later on in Section \ref{sec.ex}.
	\begin{corollary}\label{iso.Poisson.aff}
		Let $(P,\cdot,[-,-])$ and $(P',\ast,[-,-]')$ be Poisson algebras. There exists a Poisson affgebra isomorphism $\varphi \colon \paff(P;\sigma,s;\mu,\lambda,t) \rightarrow \paff(P';\sigma',s';\mu',\lambda',t')$ if and only if there exist a Poisson algebra isomorphism $\psi\colon P\rightarrow P'$ and a point $q\in P$ such that  
		\begin{subequations}\label{iso.a-f}
			\begin{gather}\label{iso_a}
				s'=\psi(s+q+q^2-q\sigma-\sigma q),\\
				\label{iso_b}
				a\sigma'=\psi(\psi^{-1}(a)\sigma-\psi^{-1}(a)q),\\
				\label{iso_c}
				\sigma' a=\psi(\sigma\psi^{-1}(a)-q\psi^{-1}(a)),\\
				\label{iso_d}
				t'=\psi(t+q-(\mu+\lambda)(q)),\\
				\label{iso_e}
				\mu'=\psi(\mu-{\rm ad}_{q})\psi^{-1},\\
				\label{iso_f}
				\lambda'=\psi(\lambda+{\rm ad}_{q})\psi^{-1},
			\end{gather}
		\end{subequations}
		for all $a\in P$.
	\end{corollary}
	\begin{proof} This result follows from Theorem~\ref{hom.Poisson.aff} by rearranging the conditions \eqref{hom.a-f}. In this situation, since $\varphi$ is an isomorphism, its linear part $\psi=\varphi-q'$, with $q'\coloneqq \varphi(0)$, is an isomorphism of vector spaces between $P$ and $P'$. Set $q\coloneqq \psi^{-1}(q')\in P.$ It is easy to compute that the following equalities hold:
		\begin{gather*}
			\psi \circ {\rm ad}_{q}\circ \psi^{-1}={\rm ad}'_{q'},\qquad
			\psi \circ (-\cdot q)\circ \psi^{-1}=-\ast q',\qquad \psi \circ (q\cdot -)\circ \psi^{-1}=q'\ast-.
		\end{gather*}
		Therefore, using the previous identities one easily rearranges conditions \eqref{hom.a-f} in Theorem~\ref{hom.Poisson.aff} to obtain conditions \eqref{iso.a-f}.  
	\end{proof} 
	
	Furthermore, it is worth emphasizing that the preceding corollary forces the fibres of two isomorphic Poisson affgebras to be isomorphic as Poisson algebras.
	
	\section{Examples and classification of Poisson affgebras}\label{sec.ex}
	
	As mentioned above, Corollary \ref{iso.Poisson.aff} provides a method to classify Poisson affgebras, up to isomorphisms, according to their dimension. In each case, we will rely on the existing dimension-based classifications of Poisson algebras in the linear setting, since Theorem \ref{main.th.poisson} establishes that every Poisson affgebra has a Poisson algebra of the same dimension as its fibre.
	
	\subsection{Classification of 1-dimensional Poisson affgebras}\label{clas1dim}
	Let $(P,\cdot,[-,-])$ be a 1-dimensional Poisson algebra with basis $\{e\}$. In this case, the only possible Lie algebra is the abelian one, i.e., $[-,-]=0$. Therefore, \eqref{Leib.rule} holds independently of the chosen associative product. Let us assume that $e\cdot e=e^{2}=\alpha e$, $\sigma=(\overleftarrow{\sigma},\overrightarrow{\sigma})$ is a double operator on $P$ with
	\[e\sigma=\overleftarrow{\sigma}(e)=le,\qquad \sigma e=\overrightarrow{\sigma}(e)=re,\]
	and $s=se$ such that $\alpha,l,r,s\in \mathbb{F}$. Here and below we use the same letter and font to denote both an element of the one-dimensional vector space {$P$} as well as its coordinate with respect to the chosen basis $\{e\}$ for {$P$}.
	\begin{lemma}
		With the notations above, $(\sigma, s)$ is a homothetic datum on $P$ if and only if one of the following cases holds: 
		\begin{itemize}
			\item $l=r$ (in other words, $\overrightarrow{\sigma}=\overleftarrow{\sigma}$) and $l^{2}-l-s\alpha=0$.  
			\item $\alpha=0$, $s=0$ with \emph{$\big[$}$l=0$ and $r=1$ (equivalently, $\overleftarrow{\sigma}=0$ and $\overrightarrow{\sigma}={\rm id}$)\emph{$\big]$} or \emph{$\big[$}$l=1$ and $r=0$ (equivalently, $\overrightarrow{\sigma}={\rm id}$ and $\overleftarrow{\sigma}=0$)\emph{$\big]$}.
		\end{itemize}
	\end{lemma}
	\begin{proof}
		For the sake of brevity, we will only verify the axioms that yield the conditions in the statement; all other axioms hold automatically. 
		
		Let $a,b\in P$ such that $a=ae$ and $b=be$ with $a,b\in\mathbb{F}.$ It is obtained that
		\begin{align*}
			a(\sigma b)&=(ae)(\sigma(be))=ab\,e(\sigma e)=abr\,e^{2}=abr\alpha\,e,
		\end{align*}
		whereas,
		\begin{align*}
			(a\sigma)b&=((ae)\sigma)(be)=ab\,(e\sigma)e=abl\,e^{2}=abl\alpha\,e.
		\end{align*}
		Then, \eqref{bimultiplication} holds if and only if $l=r$ or $\alpha=0$.
		
		If $l=r$, then 
		\begin{align*}
			a\sigma^{2}=((ae)\sigma)\sigma=al^{2}\,e
		\end{align*}
		and 
		\begin{align*}
			a(\sigma+\overline{s})=a(e\sigma+se^{2})=a(le+s\alpha e)=a(l+s\alpha)\,e.
		\end{align*}
		Consequently, \eqref{homothetic datum} holds if and only if $l^{2}=l+s\alpha$.
		
		In the case that $l\neq r$, necessarily $\alpha=0$, and then, by \eqref{commuts}, one easily obtains that $s=0$. Thus, arguing as in the previous computation, by \eqref{homothetic datum} it is obtained that $l^2-l=0=r^{2}-r$.
	\end{proof}
	
	Let us consider $\lambda,\mu\in\operatorname{Lin}(P)$ such that $\lambda(e)=\lambda e$ and $\mu(e)=\mu e$ with $\lambda,\mu\in\mathbb{F}.$ Since $[-,-]=0$, \eqref{gen_der} holds regardless of the values of $\lambda$ and $\mu$. Hence, Theorem \ref{main.th.poisson} claims that $\paff(P;\sigma, s;\mu,\lambda,t)$ is a Poisson affgebra if and only if $\lambda$ and $\mu$ are derivations on $(P,\cdot)$. There are two possible scenarios: if the associative product vanishes, that is, $\alpha=0$, then $\lambda$ and $\mu$ can be any linear map; however, in case that $\alpha\neq 0$, $\lambda=0=\mu$ whose proof is straightforward.
	
	In summary, in dimension 1 there are 5 possible families of Poisson affgebras: if $\alpha=0$, i.e., $-\cdot -=0$, the following four families arise:
	\begin{itemize}
		\item[(1A)] $\paff(P;\sigma, 0;\mu,\lambda,t)$ with $\sigma=(0,{\rm id})$,
		\item[(1B)] $\paff(P;\sigma, 0;\mu,\lambda,t)$ with $\sigma=({\rm id},0)$,
		\item[(1C)] $\paff(P;\sigma, s;\mu,\lambda,t)$ with $\sigma=(0,0)$,
		\item[(1D)] $\paff(P;\sigma, s;\mu,\lambda,t)$ with $\sigma=({\rm id},{\rm id})$,
	\end{itemize}
	whereas, if $\alpha\neq 0$, the following one is obtained:
	\begin{itemize}
		\item[(1E)] $\paff(P;\sigma, s;0,0,t)$ with $\sigma=(\overleftarrow{\sigma},\overleftarrow{\sigma})$ and $\overleftarrow{\sigma}(e)=le$ such that $l^{2}-l-s\alpha=0$. 
	\end{itemize}
	
	Let us work under the assumptions of (1A). Let $\psi\colon P\rightarrow P$ be a Poisson algebra automorphism such that $\psi(e)=\psi e$ with $0\neq\psi\in \mathbb{F}$, and $q=qe$ such that $q\in\mathbb{F}$. Applying Corollary \ref{iso.Poisson.aff} with $t=te$, $t\in\mathbb{F}$, it is obtained that $s'=0$, $\overleftarrow{\sigma'}=\overleftarrow{\sigma}=0$ and $\overrightarrow{\sigma'}=\overrightarrow{\sigma}={\rm id}$ because all maps commute in dimension 1, $\mu'=\mu$ and $\lambda'=\lambda$ for the same reason, and 
	\[t'=\psi(t+q-(\lambda+\mu)(q))=\psi[t+(1-\lambda-\mu)q]\,e.\]
	
	Therefore, if $1-\lambda-\mu\neq 0,$ then we can set $q=\frac{t}{\lambda+\mu-1}$ which yields $t'=0$. In the case that $1-\lambda-\mu= 0$, $t'=\psi t\, e$. Then, if $t\neq 0$, we obtain that $t'=e$ by setting $\psi=\frac{1}{t}$; however, if $t=0$, then $t'=0$. In brief, the possible isomorphism classes of Poisson affgebras belonging to family (1A) are the following:
	\[\textnormal{(1A)}\quad \paff(P;\sigma=(0,{\rm id}), 0;\mu,\lambda,t) \simeq \left\{\begin{array}{lcl}
		\paff(P;(0,{\rm id}), 0;\mu,\lambda,0) &\text{if} & 1-\lambda-\mu\neq 0,  \\
		\paff(P;(0,{\rm id}), 0;\mu,{\rm id}-\mu ,e) &\text{if} & \lambda=1-\mu\textnormal{ and }t\neq 0,\\
		\paff(P;(0,{\rm id}), 0;\mu,{\rm id}-\mu ,0) &\text{if} & \lambda=1-\mu\textnormal{ and }t= 0.
	\end{array}\right.\]
	An analogous argument yields the isomorphism classes for family (1B), which are the same as for (1A), but with $\sigma=({\rm id},0).$ 
	
	Now, let us study the family (1C). Applying Corollary \ref{iso.Poisson.aff} it is obtained that $s'=\psi(s+q)$. So, $s'=0$ by setting $q=-s$. Moreover, since maps in dimension 1 commute, $\overleftarrow{\sigma'}=\overleftarrow{\sigma}=0$, $\overrightarrow{\sigma'}=\overrightarrow{\sigma}=0$, $\lambda'=\lambda$ and $\mu'=\mu$. Moreover, taking $t=te$ with $t\in\mathbb{F}$, an easy computation gives 
	\[t'=\psi(t+q-(\lambda+\mu)(q))\overset{q=-s}{=}\psi(t-s+(\lambda+\mu)(s))=\psi[t+s(\lambda+\mu-1)]\,e. \]
	
	In conclusion, if $t+s(\lambda+\mu-1)\neq 0$, then $t'=e$ by setting $\psi=[t+s(\lambda+\mu-1)]^{-1}$. In contrast, when $t+s(\lambda+\mu-1)=0$, $t'=0$ regardless the value of $\psi$. In short, we now list the possible isomorphism classes of Poisson affgebras in family (1C):
	\[\textnormal{(1C)}\quad \paff(P;\sigma=(0,0), s;\mu,\lambda,t)\simeq \left\{\begin{array}{lcl}
		\paff(P;(0,0), 0;\mu,\lambda,0) &\text{if} & t+s(\lambda+\mu-1)=0,  \\
		\paff(P;(0,0), 0;\mu,\lambda,e) &\text{if} & t+s(\lambda+\mu-1)\neq 0.
	\end{array}\right.\]
	The same arguments as in (1C) apply to case (1D), except that $\sigma=({\rm id},{\rm id}).$
	
	We are left with the analysis of the family (1E) only. In this case, since $\alpha\neq 0$ and the automorphism $\psi$ has to preserve the associative product, one easily obtains that $\psi={\rm id}.$ A direct application of Corollary \ref{iso.Poisson.aff} yields the following results: $\mu'=\mu=0$ and $\lambda'=\lambda=0$, and also
	\begin{gather*}
		s'=s+q+q^{2}-2q\sigma =(s+q+q^{2}\alpha -2ql)\,e,\qquad
		a\sigma'=a\sigma-aq=a(l-\alpha q)\,e=\sigma'a,
	\end{gather*}
	Hence, setting $q=\frac{l}{\alpha}$, $\overleftarrow{\sigma'}=0=\overrightarrow{\sigma'}$ and 
	\[s'=\frac{\alpha s+l-l^{2}}{\alpha}\overset{l^{2}-l-\alpha s=0}{=}0.\]
	Moreover, $t'=t+\frac{l}{\alpha}e.$ Then, we conclude that
	\[\textnormal{(1E)}\quad \paff(P;\sigma=(\overleftarrow{\sigma},\overleftarrow{\sigma}), s;0,0,t)\simeq \paff(P;(0,0), 0;0,0,t+\frac{l}{\alpha}e).\]
	
	\subsection{Sample classification of complex 2-dimensional Poisson affgebras}\label{ssec.2-dim}
	In this section, we deal with the classification of complex Poisson affgebras in dimension 2. For this purpose, our starting point will be the well-known classification of complex 2-dimensional Poisson algebras \cite[Section~2.1]{GozRem}, since any complex 2-dimensional Poisson affgebra must admit one of these as its fibre. For the sake of completeness, let $(P,\cdot,[-,-])$ be a complex 2-dimensional Poisson algebra with basis $\{e_{1},e_{2}\}$. Two cases will be distinguished: either $[-,-]=0$, in which case the Lie part of $P$ is abelian, or $[-,-]\neq 0$. When $[-,-]=0$, the compatibility between the associative product and Poisson bracket holds automatically. Thus the classification of Poisson algebras reduces to considering four complex 2-dimensional associative algebras \cite{Pierce}, which will be detailed in Table \ref{tab: 1} (where only the nonzero basis products are listed): 
	
	\begin{table}[h]
		\begin{minipage}{0.5\linewidth}
			\begin{center}
				\begin{tabular}{|| c | c | c ||}\hline
					& Product & $\operatorname{Aut}(P)$\\\hline\hline
					(2A) & $e_{1}^{2}=e_{2}$ & $\begin{pmatrix}
						\theta & 0\\\delta&\theta^{2}
					\end{pmatrix},$ $\theta\neq 0$\\\hline
					(2B) & $\begin{array}{c} e_{1}^{2}=e_{1}\\e_{1}e_{2}=e_{2}\end{array}$ & $\begin{pmatrix}
						1 & 0\\\delta&\theta
					\end{pmatrix},$ $\theta\neq 0$\\\hline
				\end{tabular}
			\end{center}
		\end{minipage}
		\begin{minipage}{0.49\linewidth}
			\begin{center}
				\begin{tabular}{|| c | c | c ||}\hline
					& Product & $\operatorname{Aut}(P)$\\\hline\hline
					(2C) & $\begin{array}{c} e_{1}^{2}=e_{1}\\e_{2}e_{1}=e_{2}\end{array}$ & $\begin{pmatrix}
						1 & 0\\\delta&\theta
					\end{pmatrix},$ $\theta\neq 0$\\\hline
					(2D) & $\begin{array}{c} e_{1}^{2}=e_{1}\\e_{1}e_{2}=e_{2}=e_{2}e_{1}\end{array}$ & $\begin{pmatrix}
						1 & 0\\0&\theta
					\end{pmatrix},$ $\theta\neq 0$\\\hline
				\end{tabular}
			\end{center}
		\end{minipage}
		\vspace{.2cm}
		\caption{Classification of $2$-dimensional Poisson algebras with $[-,-]=0$.}
		\label{tab: 1}
	\end{table}
	
	On the other hand, if $[-,-]\neq 0,$ then it is isomorphic to the Poisson algebra described in Table~\ref{tab: 2}: 
	\begin{table}[h]
		\begin{center}
			\begin{tabular}{|| c | c | c ||}\hline
				& Product & $\operatorname{Aut}(P)$\\\hline\hline
				(2E) & $\begin{array}{c}
					-\cdot -=0   \\\, [e_{1},e_{2}]=e_{2} 
				\end{array}$ & $\begin{pmatrix}
					1 & 0\\\delta&\theta
				\end{pmatrix},$ $\theta\neq 0$\\\hline
			\end{tabular}
			\vspace{.2cm}
			\caption{Classification of $2$-dimensional Poisson algebras with $[-,-]\neq 0$.}
			\label{tab: 2}
		\end{center}
	\end{table}
	
	Only cases (2A) and (2E) will be discussed in detail, as they illustrate the techniques required to complete the remaining cases, which can be treated analogously. Following the same approach as in Section~\ref{clas1dim}, we first determine all admissible homothetic data and the conditions on $\lambda$ and $\mu$ ensuring that they are derivations of the associative product and satisfy the generalized derivation condition \eqref{gen_der}. We then employ Corollary \ref{iso.Poisson.aff} to compute all corresponding isomorphism classes.
	\subsubsection{\underline{Case \textbf{\emph{(2A):}} $\bm{[-,-]=0}$\emph{\textbf{;}} $\bm{e_{1}^{2}=e_{2}}$}}
	Let us consider a double operator $\sigma=(\overleftarrow{\sigma},\overrightarrow{\sigma})$ on $P$ given by $\overleftarrow{\sigma}=\begin{pmatrix}
		\sigma_{1}&\sigma_{3}\\\sigma_{2}&\sigma_{4}
	\end{pmatrix}$ and $\overrightarrow{\sigma}=\begin{pmatrix}
		\sigma_{5}&\sigma_{7}\\\sigma_{6}&\sigma_{8}
	\end{pmatrix}$, together with an element $s=\begin{pmatrix}
		s_{1}\\s_{2}
	\end{pmatrix}\in P$.
	\begin{lemma}
		With the notations above, $(\sigma,s)$ is a homothetic datum on $P$ if and only if one of the following cases holds: 
		
		\begin{minipage}{0.5\linewidth}
			\begin{itemize}
				\item $\overleftarrow{\sigma}=\overrightarrow{\sigma}=\begin{pmatrix}
					0&0\\\sigma_{2} & 0
				\end{pmatrix},\quad s=\begin{pmatrix}
					-\sigma_{2}\\s_{2}
				\end{pmatrix}$,
			\end{itemize}
		\end{minipage}
		\begin{minipage}{0.49\linewidth}
			\begin{itemize}
				\item $\overleftarrow{\sigma}=\overrightarrow{\sigma}=\begin{pmatrix}
					1&0\\\sigma_{2} & 1
				\end{pmatrix},\quad s=\begin{pmatrix}
					\sigma_{2}\\s_{2}
				\end{pmatrix}$.
			\end{itemize}
		\end{minipage}
	\end{lemma}
	\begin{proof}
		The proof is straightforward in view of the linearity of $\overleftarrow{\sigma}$ and $\overrightarrow{\sigma}$, as well as the definition of the associative product. We therefore only specify how each of the conditions is obtained. Firstly, bimultiplication equalities \eqref{bimultiplication} lead to 
		\[\sigma_{3}=\sigma_{7}=0,\qquad
		\sigma_{1}=\sigma_{4}=\sigma_{5}=\sigma_{8},\]
		which automatically imply that $\sigma$ is a double homothetism with no additional conditions.
		
		Moreover, \eqref{commuts} holds if and only if \begin{gather}\label{x1}s_{1}(\sigma_{2}-\sigma_{6})=0.\end{gather} Consequently, two possibilities arise:
		
		\fbox{$s_{1}\neq 0$} In this case, necessarily $\sigma_{2}=\sigma_{6}$ for \eqref{x1} to hold, i.e., $\overleftarrow{\sigma}=\overrightarrow{\sigma}=\begin{pmatrix}
			\sigma_{1}&0\\\sigma_{2} & \sigma_{1}
		\end{pmatrix}$. Therefore, under these conditions, \eqref{homothetic datum} is verified if and only if 
		\begin{gather}
			\label{x2}\sigma_{1}(\sigma_{1}-1)=0,\\
			\label{x3}2\sigma_{1}\sigma_{2}=\sigma_{2}+s_{1}.
		\end{gather}
		Note that \eqref{x2} gives rise to two possible subcases: $\sigma_{1}=0$ or $\sigma_{1}=1$. On the one hand, if $\sigma_{1}=0$, then \eqref{x3} implies that $s_{1}=-\sigma_{2}.$ On the other hand, when $\sigma_{1}=1$, we obtain that $s_{1}=\sigma_{2}$.
		
		Thus, we are left with two possible families:
		\begin{gather}\label{f1A}
			\overleftarrow{\sigma}=\overrightarrow{\sigma}=\begin{pmatrix}
				0&0\\\sigma_{2} & 0
			\end{pmatrix},\quad s=\begin{pmatrix}
				-\sigma_{2}\\s_{2}
			\end{pmatrix} \textnormal{ such that }\sigma_{2}\neq0,
		\end{gather}
		and 
		\begin{gather}\label{f1B}
			\overleftarrow{\sigma}=\overrightarrow{\sigma}=\begin{pmatrix}
				1&0\\\sigma_{2} & 1
			\end{pmatrix},\quad s=\begin{pmatrix}
				\sigma_{2}\\s_{2}
			\end{pmatrix} \textnormal{ such that }\sigma_{2}\neq0.
		\end{gather}
		
		\fbox{$s_{1}= 0$} Under this assumption, $\overleftarrow{\sigma}=\begin{pmatrix}
			\sigma_{1}&0\\\sigma_{2} & \sigma_{1}
		\end{pmatrix}$ and $\overrightarrow{\sigma}=\begin{pmatrix}
			\sigma_{1}&0\\\sigma_{6}&\sigma_{1}
		\end{pmatrix}$ with $s=\begin{pmatrix}
			0\\s_{2}
		\end{pmatrix}$. Arguing as above, by considering \eqref{homothetic datum} with $\overleftarrow{\sigma}$, it is satisfied if and only if \eqref{x2} and 
		\begin{gather}
			\label{x4} 2\sigma_{1}\sigma_{2}=\sigma_{2}
		\end{gather}
		hold. Hence, regardless of whether $\sigma_{1}$ is 0 or 1, \eqref{x4} implies that $\sigma_{2}=0$. Now, again independently of $\sigma_{1}\in\{0,1\}$, by computing \eqref{homothetic datum} with $\overrightarrow{\sigma}$, this equality leads to $\sigma_{6}=0$. In summary, two families arise:
		\[\overleftarrow{\sigma}=\overrightarrow{\sigma}=0,\: s=\begin{pmatrix}
			0\\s_{2}
		\end{pmatrix},\quad\textnormal{ or }\quad \overleftarrow{\sigma}=\overrightarrow{\sigma}={\rm id},\: s=\begin{pmatrix}
			0\\s_{2}
		\end{pmatrix}.\]
		
		Since the first subfamily above is obtained by setting $\sigma_{2}=0$ in \eqref{f1A}, while the second one is obtained by setting $\sigma_{2}=0$ in \eqref{f1B}, the four subfamilies arising in the proof reduce to the two families appearing in the statement. This remark completes the proof.
	\end{proof}
	
	In other words, the previous lemma asserts that there are only two possible families of admissible homothetic data on $P$. 
	
	Let us now consider linear endomorphism of $P$, $\lambda=\begin{pmatrix}
		\lambda_{1}&\lambda_{3}\\\lambda_{2}&\lambda_{4}
	\end{pmatrix}$ and $\mu=\begin{pmatrix}
		\mu_{1}&\mu_{3}\\\mu_{2}&\mu_{4}
	\end{pmatrix}$. Since $[-,-]=0$, \eqref{gen_der} always holds independently of the choice of $\lambda$ and $\mu$. Then, according to Theorem~\ref{main.th.poisson}, $\paff(P;\sigma, s;\mu,\lambda,t)$ defines a Poisson affgebra if and only if $\lambda$ and $\mu$ are derivations on $(P,\cdot)$. Direct computation shows that this happens if and only if $\lambda_{3}=0$ and $\lambda_{4}=2\lambda_{1}$, and the same for $\mu$: $\mu_{3}=0$, $\mu_{4}=2\mu_{1}$.
	In conclusion, there are two possible families of Poisson affgebras with fibre the Poisson algebra (2A):
	\begin{itemize}
		\item[(FA.1)] $\paff(P;\sigma, s;\mu,\lambda,t)$ such that $\overleftarrow{\sigma}=\overrightarrow{\sigma}=\begin{pmatrix}
			0&0\\\sigma_{2} & 0
		\end{pmatrix}$, $s=\begin{pmatrix}
			-\sigma_{2}\\s_{2}
		\end{pmatrix}$,
		\item[(FA.2)] $\paff(P;\sigma, s;\mu,\lambda,t)$ such that $\overleftarrow{\sigma}=\overrightarrow{\sigma}=\begin{pmatrix}
			1&0\\\sigma_{2} & 1
		\end{pmatrix}$, $s=\begin{pmatrix}
			\sigma_{2}\\s_{2}
		\end{pmatrix}$,
	\end{itemize}
	where $\lambda=\begin{pmatrix}
		\lambda_{1}&0\\\lambda_{2} &2\lambda_{1}
	\end{pmatrix}$ and $\mu=\begin{pmatrix}
		\mu_{1}&0\\\mu_{2} &2\mu_{1}
	\end{pmatrix}$, with $t=\begin{pmatrix}
		t_{1}\\t_{2}
	\end{pmatrix}$ an element in $P$. We restrict our analysis of isomorphism classes to family (FA.1), as family (FA.2) can be treated analogously.
	
	Applying Corollary \ref{iso.Poisson.aff} to (FA.1), by \eqref{iso_a} one easily obtains that $s'=\psi\begin{pmatrix}
		-\sigma_{2}+q_{1}\\s_{2}+q_{2}+q_{1}^{2}-2\sigma_{2}q_{1}
	\end{pmatrix}.$ Therefore, setting $q_{1}=\sigma_{2}$ and $q_{2}=\sigma_{2}^{2}-s_{2}$ leads to $s'=0$. Let us consider $\psi=\begin{pmatrix}
		\theta&0\\\delta&\theta^{2}
	\end{pmatrix}$, $\theta\neq 0$, an automorphism of $P$ whose inverse is $\psi^{-1}=\frac{1}{\theta^{3}}\begin{pmatrix}
		\theta^{2}&0\\-\delta &\theta
	\end{pmatrix}$. By \eqref{iso_b} and \eqref{iso_c}, one obtains that 
	\[\overleftarrow{\sigma'}=\begin{pmatrix}
		0&0\\\theta(\sigma_{2}-q_{1})&0
	\end{pmatrix}=\overrightarrow{\sigma'}\:\overset{q_{1}=\sigma_{2}}{\implies}\:\overleftarrow{\sigma'}=0=\overrightarrow{\sigma'},\]
	and it is important to note that this step is independent of the automorphism $\psi$. 
	
	Moreover, taking into account that ${\rm ad}_{q}=0$ since $[-,-]=0$, by \eqref{iso_d}, \eqref{iso_e} and \eqref{iso_f} it results that
	\begin{gather}\nonumber\mu'=\begin{pmatrix}
			\mu_{1}&0\\\frac{\theta^{2}\mu_{2}-\delta\mu_{1}}{\theta}&2\mu_{1}
		\end{pmatrix},\qquad \lambda'=\begin{pmatrix}
			\lambda_{1}&0\\\frac{\theta^{2}\lambda_{2}-\delta\lambda_{1}}{\theta}&2\lambda_{1}
		\end{pmatrix},\\ \label{tprime_general}t'=\begin{pmatrix}
			\theta[t_{1}+\sigma_{2}(1-\mu_{1}-\lambda_{1})]\\
			\delta[t_{1}+\sigma_{2}(1-\mu_{1}-\lambda_{1})]+\theta^{2}[t_{2}+(\sigma_{2}^{2}-s_{2})(1-2\mu_{1}-2\lambda_{1})-\sigma_{2}(\mu_{2}+\lambda_{2})]
		\end{pmatrix}.\end{gather}
	
	If $\mu_{1}\neq 0$, then $\mu'=\textnormal{diag}(\mu_{1},2\mu_{1})$ by setting $\delta=\frac{\theta^{2}\mu_{2}}{\mu_{1}}.$ With this choice for $\delta$, \[\lambda'=\begin{pmatrix}
		\lambda_{1}&0\\\frac{\theta}{\mu_{1}}(\lambda_{2}\mu_{1}-\mu_{2}\lambda_{1})&2\lambda_{1}
	\end{pmatrix}.\] Thus, if $\lambda_{2}\mu_{1}-\mu_{2}\lambda_{1}\neq 0$, then $\lambda'=\begin{pmatrix}
		\lambda_{1}&0\\1&2\lambda_{1}
	\end{pmatrix}$ taking $\theta=\frac{\mu_{1}}{\lambda_{2}\mu_{1}-\mu_{2}\lambda_{1}}$. Hence, upon substituting these values of $\theta$ and $\delta$ in \eqref{tprime_general}, we obtain that
	{\small
		\begin{gather}\label{tprime1}
			t'=\begin{pmatrix}
				\frac{\mu_{1}}{\lambda_{2}\mu_{1}-\mu_{2}\lambda_{1}}\left[t_{1}+\sigma_{2}(1-\mu_{1}-\lambda_{1})\right]\\
				\frac{\mu_{1}}{(\lambda_{2}\mu_{1}-\mu_{2}\lambda_{1})^{2}}\left(\mu_{2}[t_{1}+\sigma_{2}(1-\mu_{1}-\lambda_{1})]+\mu_{1}[t_{2}+(\sigma_{2}^{2}-s_{2})(1-2\mu_{1}-2\lambda_{1})-\sigma_{2}(\mu_{2}+\lambda_{2})]\right)
			\end{pmatrix}.
	\end{gather}}
	
	Whereas, if $\lambda_{2}\mu_{1}-\mu_{2}\lambda_{1}= 0$, then $\lambda'=\textnormal{diag}(\lambda_{1},2\lambda_{1})$ and $t'$ becomes
	\begin{gather*}
		t'=\begin{pmatrix}
			\theta[t_{1}+\sigma_{2}(1-\mu_{1}-\lambda_{1})]\\
			\frac{\theta^{2}}{\mu_{1}}[\mu_{2}t_{1}+\mu_{1}t_{2}+(1-2\mu_{1}-2\lambda_{1})(\mu_{1}(\sigma^{2}_{2}-s_{2})+\sigma_{2}\mu_{2})]
		\end{pmatrix}.
	\end{gather*}
	So, in the case that $t_{1}+\sigma_{2}(1-\mu_{1}-\lambda_{1})\neq 0$, it is possible to take $\theta=[t_{1}+\sigma_{2}(1-\mu_{1}-\lambda_{1})]^{-1}$, and then 
	\begin{gather}\label{tprime2}
		t'=\begin{pmatrix}
			1\\ \frac{\mu_{2}t_{1}+\mu_{1}t_{2}+(1-2\mu_{1}-2\lambda_{1})(\mu_{1}(\sigma^{2}_{2}-s_{2})+\sigma_{2}\mu_{2})}{\mu_{1}[t_{1}+\sigma_{2}(1-\mu_{1}-\lambda_{1})]^{2}}
		\end{pmatrix}.
	\end{gather}
	However, if $t_{1}=\sigma_{2}(\mu_{1}+\lambda_{1}-1)$, then $t'=\theta^{2}De_{2}$, where $D\coloneqq t_{2}+(1-2\mu_{1}-2\lambda_{1})(\sigma_{2}^{2}-s_{2})-\frac{\sigma_{2}\mu_{2}}{\mu_{1}}(\lambda_{1}+\mu_{1}).$ Then, two possible cases arise: $t'=0$ if $D=0$, or $t'=\Omega_{D}e_{2}$, where $\Omega_{D}$ is any element in the square class of $D$, when $D\neq 0$.
	
	If $\mu_{1}=0$, then $\mu'=\begin{pmatrix}
		0&0\\\theta\mu_{2}&0
	\end{pmatrix}$. In case $\mu_{2}\neq 0$, then $\mu'=\begin{pmatrix}
		0&0\\1&0
	\end{pmatrix}$ by setting $\theta=\mu_{2}^{-1}$. Such a choice of $\theta$ yields $\lambda'=\begin{pmatrix}
		\lambda_{1}&0\\\mu_{2}^{-1}(\lambda_{2}-\delta\mu_{2}^{2}\lambda_{1})&2\lambda_{1}
	\end{pmatrix}$. Now assuming that $\lambda_{1}\neq 0$, it is possible to set $\delta=\frac{\lambda_{2}}{\mu_{2}^{2}\lambda_{1}}$ which implies that $\lambda'=\textnormal{diag}(\lambda_{1},2\lambda_{1})$, and then
	\begin{gather}\label{tprime3}
		t'=\begin{pmatrix}
			\mu_{2}^{-1}[t_{1}+\sigma_{2}(1-\lambda_{1})]\\
			\frac{1}{\mu_{2}^{2}\lambda_{1}}\left(\lambda_{2}[t_{1}+\sigma_{2}(1-\lambda_{1})]+\lambda_{1}[t_{2}+(\sigma_{2}^{2}-s_{2})(1-2\lambda_{1})-\sigma_{2}(\mu_{2}+\lambda_{2})]\right)
		\end{pmatrix}.
	\end{gather}
	
	Conversely, if $\lambda_{1}=0$, then $\lambda'=\begin{pmatrix}
		0&0\\\frac{\lambda_{2}}{\mu_{2}}&0
	\end{pmatrix}$ and $t'=\begin{pmatrix}
		\mu_{2}^{-1}(t_{1}+\sigma_{2})\\\delta(t_{1}+\sigma_{2})+\mu_{2}^{-2}[t_{2}+\sigma_{2}^{2}-s_{2}-\sigma_{2}(\lambda_{2}+\mu_{2})]
	\end{pmatrix}$. Again, this situation splits into two subcases: on the one hand, if $t_{1}+\sigma_{2}\neq 0$, then one can set $\delta=\frac{\mu_{2}^{-2}[\sigma_{2}(\lambda_{2}+\mu_{2})+s_{2}-\sigma_{2}^{2}-t_{2}]}{t_{1}+\sigma_{2}}$ obtaining that $t'=\mu_{2}^{-1}(t_{1}+\sigma_{2})e_{1}$; on the other hand, if $t_{1}=-\sigma_{2}$, then 
	\begin{equation}\label{tprime5}t'=\mu_{2}^{-2}[t_{1}^{2}+t_{2}-s_{2}+t_{1}(\lambda_{2}+\mu_{2})]e_{2}.\end{equation}
	
	It only remains to study the case $\mu_{1}=0=\mu_{2}$, in which $\mu'=0$. Then, proceeding analogously:
	\begin{itemize}
		\item If $\lambda_{1}\neq 0$, then we can set $\delta=\frac{\theta^{2}\lambda_{2}}{\lambda_{1}}$ in such a way that $\lambda'=\textnormal{diag}(\lambda_{1},2\lambda_{1})$ and 
		\begin{gather*}
			t'=\begin{pmatrix}
				\theta[t_{1}+\sigma_{2}(1-\lambda_{1})]\\
				\theta^{2}[t_{2}+(\sigma_{2}^{2}-s_{2})(1-2\lambda_{1})+\frac{\lambda_{2}}{\lambda_{1}}(t_{1}+\sigma_{2}(1-2\lambda_{1}))]
			\end{pmatrix}.
		\end{gather*}
		\begin{itemize}
			\item If $t_{1}+\sigma_{2}(1-\lambda_{1})\neq 0$, then we obtain that 
			\begin{gather}\label{tprime4}
				t'=\begin{pmatrix}
					1\\\frac{t_{2}+(\sigma_{2}^{2}-s_{2})(1-2\lambda_{1})+\lambda_{2}\lambda_{1}^{-1}(t_{1}+\sigma_{2}(1-2\lambda_{1}))}{[t_{1}+\sigma_{2}(1-\lambda_{1})]^{2}}
				\end{pmatrix}.
			\end{gather}
			by setting $\theta=[t_{1}+\sigma_{2}(1-\lambda_{1})]^{-1}$.
			\item If $t_{1}=\sigma_{2}(\lambda_{1}-1)$, then $t'=\theta^{2}E\,e_{2}$, where $E\coloneqq t_{2}+(\sigma_{2}^{2}-s_{2})(1-2\lambda_{1})-\sigma_{2}\lambda_{2}$. Hence, in case $E$ vanishes, then $t'=0$; however, if $E\neq 0$, then $t'=\Omega_{E}e_{2}$ with $\Omega_{E}$ any element in the square class of $E$.
		\end{itemize}

		\item If $\lambda_{1}=0$ and $\lambda_{2}\neq 0$, then taking $\theta=\lambda_{2}^{-1}$ one obtains that $\lambda'=\begin{pmatrix}
			0&0\\1&0
		\end{pmatrix}$ and also $t'$ becomes
		\begin{gather*}
			t'=\begin{pmatrix}
				\lambda_{2}^{-1}(t_{1}+\sigma_{2})\\
				\delta(t_{1}+\sigma_{2})+\frac{t_{2}+\sigma_{2}^{2}-s_{2}-\sigma_{2}\lambda_{2}}{\lambda_{2}^{2}}
			\end{pmatrix}.
		\end{gather*}
		Therefore, if $t_{1}+\sigma_{2}\neq 0$, then one can set $\delta=\frac{s_{2}+\sigma_{2}\lambda_{2}-\sigma_{2}^{2}-t_{2}}{\lambda_{2}^{2}(t_{1}+\sigma_{2})}$ in order to obtain that $t'=\lambda_{2}^{-1}(t_{1}+\sigma_{2})e_{1}$. However, if $t_{1}=-\sigma_{2}$, then $t'=\frac{t_{2}+t_{1}^{2}-s_{2}+t_{1}\lambda_{2}}{\lambda_{2}^{2}}e_{2}$.

		\item If $\lambda_{1}=0=\lambda_{2},$ then $\lambda'=0$ and $t'=\begin{pmatrix}
			\theta(t_{1}+\sigma_{2})\\\delta(t_{1}+\sigma_{2})+\theta^{2}(t_{2}+\sigma_{2}^{2}-s_{2})
		\end{pmatrix}$. So, if $t_{1}+\sigma_{2}\neq 0$, it is obtained that $t'=e_{1}$ by setting $\delta=\frac{\theta^{2}(s_{2}-t_{2}-\sigma_{2}^{2})}{t_{1}+\sigma_{2}}$ and $\theta=(t_{1}+\sigma_{2})^{-1}$.
		
		In case that $t_{1}=-\sigma_{2}$, $t'=\theta^{2} F e_{2} $, with $F\coloneqq t_{2}+\sigma_{2}^{2}-s_{2}$. Therefore, $t'$ will vanish when $F=0$, but, if $F\neq 0$, then $t'=\Omega_{F}e_{2}$, where $\Omega_{F}$ is any element in the square class of $F$. 
	\end{itemize}
	
	In conclusion, 
	{\small
		\[\rotatebox{90}{$\begin{array}{c}\hspace{-1.3cm}\textnormal{(FA.1)}\\\hspace{-1.3cm}\paff(P;\sigma, s;\mu,\lambda,t)\end{array}$}\simeq \left\{ \begin{array}{l c c}

			\paff\left(P;\sigma', s';\mu'=\operatorname{diag}(\mu_{1},2\mu_{1}),\lambda'=\begin{pmatrix}
				\lambda_{1}&0\\1&2\lambda_{1}
			\end{pmatrix},t'=\eqref{tprime1}\right) & \textnormal{if} & \begin{array}{c}\mu_{1}\neq 0,\\ \lambda_{2}\mu_{1}-\mu_{2}\lambda_{1}\neq 0,\end{array}\vspace{.1cm}\\
			
			\paff\left(P;\sigma', s';\mu'=\operatorname{diag}(\mu_{1},2\mu_{1}),\lambda'=\operatorname{diag}(\lambda_{1},2\lambda_{1}),t'=\eqref{tprime2}\right) & \textnormal{if} & \begin{array}{c}\mu_{1}\neq 0,\\ \lambda_{2}\mu_{1}-\mu_{2}\lambda_{1}= 0,\\t_{1}+\sigma_{2}(1-\mu_{1}-\lambda_{1})\neq 0,\end{array}\vspace{.1cm}\\
			
			\paff\left(P;\sigma', s';\mu'=\operatorname{diag}(\mu_{1},2\mu_{1}),\lambda'=\operatorname{diag}(\lambda_{1},2\lambda_{1}),t'=\Omega_{D}e_{2}\right) & \textnormal{if} & \begin{array}{c}\mu_{1}\neq 0,\\ \lambda_{2}\mu_{1}-\mu_{2}\lambda_{1}= 0,\\t_{1}=\sigma_{2}(\mu_{1}+\lambda_{1}-1),\\D\neq 0\end{array}\vspace{.1cm}\\
			
			\paff\left(P;\sigma', s';\mu'=\operatorname{diag}(\mu_{1},2\mu_{1}),\lambda'=\operatorname{diag}(\lambda_{1},2\lambda_{1}),t'=0\right) & \textnormal{if} & \begin{array}{c}\mu_{1}\neq 0,\\ \lambda_{2}\mu_{1}-\mu_{2}\lambda_{1}= 0,\\t_{1}=\sigma_{2}(\mu_{1}+\lambda_{1}-1),\\D= 0\end{array}\vspace{.1cm}\\
			
			\paff\left(P;\sigma', s';\mu'=\begin{pmatrix}
				0&0\\1&0
			\end{pmatrix},\lambda'=\operatorname{diag}(\lambda_{1},2\lambda_{1}),t'=\eqref{tprime3}\right) & \textnormal{if} & \begin{array}{c} \mu_{1}= 0,\:\mu_{2}\neq 0\\ \lambda_{1}\neq  0,\end{array}\vspace{.1cm}\\
			
			\paff\left(P;\sigma', s';\mu'=\begin{pmatrix}
				0&0\\1&0
			\end{pmatrix},\lambda'=\begin{pmatrix}
				0&0\\\frac{\lambda_{2}}{\mu_{2}}&0
			\end{pmatrix},t'=\mu_{2}^{-1}(t_{1}+\sigma_{2})e_{1}\right) & \textnormal{if} & \begin{array}{c} \mu_{1}= 0,\:\mu_{2}\neq 0\\ \lambda_{1}=  0,\: t_{1}+\sigma_{2}\neq 0,\end{array}\vspace{.1cm}\\
			
			\paff\left(P;\sigma', s';\mu'=\begin{pmatrix}
				0&0\\1&0
			\end{pmatrix},\lambda'=\begin{pmatrix}
				0&0\\\frac{\lambda_{2}}{\mu_{2}}&0
			\end{pmatrix},t'=\eqref{tprime5}\right) & \textnormal{if} & \begin{array}{c} \mu_{1}= 0,\:\mu_{2}\neq 0\\ \lambda_{1}=  0,\: t_{1}=-\sigma_{2},\end{array}\vspace{.1cm}\\
			
			\paff\left(P;\sigma', s';\mu'=0,\lambda'=\textnormal{diag}(\lambda_{1},2\lambda_{1}),t'=\eqref{tprime4}\right) & \textnormal{if} & \begin{array}{c} \mu_{1}= 0=\mu_{2}\\ \lambda_{1}\neq  0,\\t_{1}+\sigma_{2}(1-\lambda_{1})\neq 0,\end{array}\vspace{.1cm}\\
			
		\end{array}\right.\]}
	
	{\small
		\[\rotatebox{90}{$\begin{array}{c}\hspace{-1.3cm}\textnormal{(FA.1)}\\\hspace{-1.3cm}\paff(P;\sigma, s;\mu,\lambda,t)\end{array}$}\simeq \left\{ \begin{array}{l c c}
			
			\paff\left(P;\sigma', s';\mu'=0,\lambda'=\textnormal{diag}(\lambda_{1},2\lambda_{1}),t'=\Omega_{E}e_{2}\right) & \textnormal{if} & \begin{array}{c} \mu_{1}= 0=\mu_{2}\\ \lambda_{1}\neq  0,\\t_{1}=\sigma_{2}(\lambda_{1}-1),\: E\neq 0,\end{array}\vspace{.1cm}\\
			
			\paff\left(P;\sigma', s';\mu'=0,\lambda'=\textnormal{diag}(\lambda_{1},2\lambda_{1}),t'=0\right) & \textnormal{if} & \begin{array}{c} \mu_{1}= 0=\mu_{2}\\ \lambda_{1}\neq  0,\\t_{1}=\sigma_{2}(\lambda_{1}-1),\: E=0,\end{array}\vspace{.1cm}\\

			\paff\left(P;\sigma', s';\mu'=0,\lambda'=\begin{pmatrix}
				0&0\\1&0
			\end{pmatrix},t'=\lambda_{2}^{-1}(t_{1}+\sigma_{2})e_{1}\right) & \textnormal{if} & \begin{array}{c} \mu_{1}= 0=\mu_{2}\\ \lambda_{1}= 0,\:\lambda_{2}\neq 0,\\t_{1}+\sigma_{2}\neq 0,\end{array}\vspace{.1cm}\\
			
			\paff\left(P;\sigma', s';\mu'=0,\lambda'=\begin{pmatrix}
				0&0\\1&0
			\end{pmatrix},t'=\frac{t_{2}+t_{1}^{2}-s_{2}+t_{1}\lambda_{2}}{\lambda_{2}^{2}}e_{2}\right) & \textnormal{if} & \begin{array}{c} \mu_{1}= 0=\mu_{2}\\ \lambda_{1}= 0,\:\lambda_{2}\neq 0,\\t_{1}=-\sigma_{2},\end{array}\vspace{.1cm}\\
			
			\paff\left(P;\sigma', s';\mu'=0,\lambda'=0,t'=e_{1}\right) & \textnormal{if} & \begin{array}{c} \mu_{1}= 0=\mu_{2}\\ \lambda_{1}=  0=\lambda_{2},\\t_{1}+\sigma_{2}\neq0,\end{array}\vspace{.1cm}\\
			
			\paff\left(P;\sigma', s';\mu'=0,\lambda'=0,t'=\Omega_{F}e_{2}\right) & \textnormal{if} & \begin{array}{c} \mu_{1}= 0=\mu_{2}\\ \lambda_{1}=  0=\lambda_{2},\\t_{1}=-\sigma_{2},\:F\neq 0,\end{array}\vspace{.1cm}\\
			
			\paff\left(P;\sigma', s';\mu'=0,\lambda'=0,t'=0\right) & \textnormal{if} & \begin{array}{c} \mu_{1}= 0=\mu_{2}\\ \lambda_{1}=  0=\lambda_{2},\\t_{1}=-\sigma_{2},\:F= 0,\end{array}
		\end{array}\right.\]}
	where, in all the isomorphism classes, $\sigma'=(0,0)$ and $s'=0$.
	
	\subsubsection{\underline{Case \textbf{\emph{(2E):}} $\bm{-\cdot-=0}$\emph{\textbf{;}} $\bm{[e_{1},e_{2}]=e_{2}}$}}
	When $-\cdot -=0$, determining all admissible homothetic data (recall Definitions \ref{Def: dp} and \ref{Def: hd}) is equivalent to determining pairs of commuting idempotent endomorphisms that agree at a vector. 
	\begin{lemma}\label{idemp 2x2}
		Let $A\in\mathcal{M}_{2}(\mathbb{F})$ be a $2\times 2$ matrix with coefficients in $\mathbb{F}$. The matrix $A$ is idempotent, i.e., $A^{2}=A$, if and only if {$A\in\{0,I_{2}\}$ or} $A=\begin{pmatrix}
			a & c\\b&1-a
		\end{pmatrix}$ such that $a(1-a)=bc$ for any $a,b,c\in\mathbb{F}$, {that is, $\operatorname{tr}(A)=1$ and $\operatorname{det}(A)=0$}.
	\end{lemma}
	\begin{proof}
		{
			Let $A\in \mathcal{M}_{2}(\mathbb{F})$ be a non-trivial idempotent, i.e., $A\notin \{0,I_{2}\}$. In such a case, its minimal polynomial is $\mu_{A}(x)=x(x-1)$, and then, its eigenvalues are 0 and 1. Consequently, $\operatorname{tr}(A)=1$ (the sum of the eigenvalues) and $\operatorname{det}(A)=0$ (the product of the eigenvalues).}
		
		{
			Conversely, if $A\in \{0,I_{2}\}$, then $A$ is clearly idempotent. Let us consider $A\in \mathcal{M}_{2}(\mathbb{F})\backslash\{0,I_{2}\}$ with $\operatorname{tr}(A)=1$ and $\operatorname{det}(A)=0$. For any $2\times 2$ matrix, its characteristic polynomial has the following form: $P_{A}(x)=x^{2}-\operatorname{tr}(A)x+\operatorname{det}(A)$. Therefore, by the Cayley-Hamilton theorem, the following is obtained:
			\[0=A^{2}-\operatorname{tr}(A)A+\operatorname{det}(A)I_{2}\overset{\operatorname{tr}(A)=1;\,\det(A)=0}{=}A^{2}-A,\]
			concluding that $A$ is idempotent.
		}
	\end{proof}
	\begin{lemma}\label{comm idemp}
		{Let $A,B\in\mathcal{M}_{2}(\mathbb{F})\backslash\{0,I_{2}\}$ be idempotents.} If $AB=BA$, then $A=B$ or $B=I_{2}-A$.
	\end{lemma}
	\begin{proof}
		Since $A\in\mathcal{M}_{2}(\mathbb{F})$ is a non-trivial idempotent, we know from the previous lemma that $\mu_{A}(x)=x(x-1)$. Therefore, $A$ is diagonalizable with eigenvalues $0$ and $1$. The same occurs for $B$.
		
		A classical well-known result asserts that if $A$ and $B$ are commuting diagonalizable matrices (as in this case), then there exists a common basis in which they are diagonal. Hence, in such a basis, let us assume that $A=\operatorname{diag}(1,0)$. Now, $B$ must also be diagonal on that same basis, say $B=\operatorname{diag}(\alpha,\beta)$. Since $\operatorname{tr}(B)=1=\alpha+\beta$, we have only two possibilities for the matrix $B$:
		\begin{itemize}
			\item $B=\operatorname{diag}(1,0)$, in which case $A=B$,
			\item or $B=\operatorname{diag}(0,1)$, in which case $A=I_{2}-B$.\qedhere
		\end{itemize}
	\end{proof}
	
	Therefore, as a consequence of the previous lemmas, we obtain that there are {eleven} possible families of homothetic data on $P$, as it is summarized in the following result.
	\begin{proposition}
		Let $(\sigma,s)$ be a homothetic datum on $P$. Then, $(\sigma,s)$ belongs to one of the following families:
		\begin{itemize}
			\item {$\overleftarrow{\sigma}=0=\overrightarrow{\sigma}$, with no conditions on $s$.}
			\item {$\overleftarrow{\sigma}=\rm{id}=\overrightarrow{\sigma}$, with no conditions on $s$.}
			\item {$\overleftarrow{\sigma}=0$, $\overrightarrow{\sigma}=\rm{id}$ and $s=0$ (or $\overleftarrow{\sigma}=\rm{id}$, $\overrightarrow{\sigma}=0$ and $s=0$).}
			\item {$\overleftarrow{\sigma}=0$, $\overrightarrow{\sigma}$ a non-trivial idempotent and $s\in\ker(\overrightarrow{\sigma})$ (or $\overleftarrow{\sigma}$ a non-trivial idempotent, $\overrightarrow{\sigma}=0$ and $s\in\ker(\overleftarrow{\sigma})$).}
			\item {$\overleftarrow{\sigma}=\rm{id}$, $\overrightarrow{\sigma}$ a non-trivial idempotent and $s\in\ker(\overrightarrow{\sigma}-\rm{id})$ (or $\overleftarrow{\sigma}$ a non-trivial idempotent, $\overrightarrow{\sigma}=\rm{id}$ and $s\in\ker(\overleftarrow{\sigma}-\rm{id})$).}
			\item {$\overleftarrow{\sigma}$ a non-trivial idempotent, $\overrightarrow{\sigma}=\rm{id}-\overleftarrow{\sigma}$ and $s=0$ (or $\overrightarrow{\sigma}$ a non-trivial idempotent, $\overleftarrow{\sigma}=\rm{id}-\overrightarrow{\sigma}$ and $s=0$).}
			\item {$\overleftarrow{\sigma}=\overrightarrow{\sigma}$ the same non-trivial idempotent, with no conditions on $s$.}
		\end{itemize}
	\end{proposition}

	Furthermore, according to Theorem \ref{main.th.poisson}, constructing a Poisson affgebra with fibre $P$ also requires the existence of linear endomorphisms $\lambda=\begin{pmatrix}
		\lambda_{1}&\lambda_{3}\\\lambda_{2}&\lambda_{4}
	\end{pmatrix}$ and $\mu=\begin{pmatrix}
		\mu_{1}&\mu_{3}\\\mu_{2}&\mu_{4}
	\end{pmatrix}$ giving rise to a generalized derivation of the Lie algebra $(P,[-,-])$, and derivations of the associative algebra $(P,\cdot)$. Since $-\cdot -=0$, the latter condition is automatically satisfied by every linear map. Therefore, it suffices to verify the conditions under which \eqref{gen_der} is satisfied. Exploiting the bilinearity of the Lie bracket and the linearity of $\lambda$ and $\mu$, one readily obtains that
	\begin{center}
		$\lambda$ and $\mu$ satisfy \eqref{gen_der} $\iff$  $\lambda=\begin{pmatrix}
			\lambda_{1} & 0\\\lambda_{2}&\lambda_{4}
		\end{pmatrix}$, $\mu=\begin{pmatrix}
			0&0\\-\lambda_{2} & \lambda_{1}-\lambda_{4}
		\end{pmatrix}$.
	\end{center}
	
	As a consequence of the preceding analysis, we have obtained {eleven} families of admissible Poisson affgebras with fibre the Poisson algebra (2E):
	\begin{itemize}
		\item[(FE.1)] {$\paff(P;\sigma=(0,0), s;\mu,\lambda,t)$.}
		\item[(FE.2)] {$\paff(P;\sigma=({\rm id},{\rm id}), s;\mu,\lambda,t)$.}
		\item[(FE.3)] {$\paff(P;\sigma=(0,\rm{id}), 0;\mu,\lambda,t)$.}
		\item[(FE.4)] {$\paff(P;\sigma=({\rm id},0), 0;\mu,\lambda,t)$.}
		\item[(FE.5)] {$\paff(P;\sigma=(0,\overrightarrow{\sigma}), s;\mu,\lambda,t)$, where $\overrightarrow{\sigma}$ is a non-trivial idempotent and $s\in\ker(\overrightarrow{\sigma})$.}
		\item[(FE.6)] {$\paff(P;\sigma=(\overleftarrow{\sigma},0), s;\mu,\lambda,t)$, where $\overleftarrow{\sigma}$ is a non-trivial idempotent and $s\in\ker(\overleftarrow{\sigma})$.}
		\item[(FE.7)] {$\paff(P;\sigma=({\rm id},\overrightarrow{\sigma}), s;\mu,\lambda,t)$, where $\overrightarrow{\sigma}$ is a non-trivial idempotent and $s\in\ker(\overrightarrow{\sigma}-{\rm id})$.}
		\item[(FE.8)] {$\paff(P;\sigma=(\overleftarrow{\sigma},{\rm id}), s;\mu,\lambda,t)$, where $\overleftarrow{\sigma}$ is a non-trivial idempotent and $s\in\ker(\overleftarrow{\sigma}-{\rm id})$.}
		\item[(FE.9)] {$\paff(P;\sigma=(\overleftarrow{\sigma},{\rm id}-\overleftarrow{\sigma}), 0;\mu,\lambda,t)$, where $\overleftarrow{\sigma}=\begin{pmatrix}
				\sigma_{1}&\sigma_{3}\\\sigma_{2}&1-\sigma_{1}
			\end{pmatrix}$ such that $\sigma_{1}(1-\sigma_{1})=\sigma_{2}\sigma_{3}$.}
		\item[(FE.10)] {$\paff(P;\sigma=({\rm id}-\overrightarrow{\sigma},\overrightarrow{\sigma}), 0;\mu,\lambda,t)$, where $\overrightarrow{\sigma}$ is a non-trivial idempotent.}
		\item[(FE.11)] {$\paff(P;\sigma=(\overleftarrow{\sigma},\overleftarrow{\sigma}), s;\mu,\lambda,t)$, where $\overleftarrow{\sigma}$ is a non-trivial idempotent.}
	\end{itemize}
	
	For illustrative purposes, we carry out the classification up to isomorphisms for the family {(FE.9)}, {thereby justifying the parametrized form of $\overleftarrow{\sigma}$ used in the list above, whose parameters will be employed in the classification below.} In view of the characterization of isomorphisms established in Corollary \ref{iso.Poisson.aff} and the fact that $\rm{ad}_{q}=\begin{pmatrix}
		0&0\\q_{2}&-q_{1}
	\end{pmatrix}$, if $\psi=\begin{pmatrix}
		1&0\\
		\delta&\theta
	\end{pmatrix}\in\operatorname{Aut}(P)$ with inverse $\psi^{-1}=\theta^{-1}\begin{pmatrix}
		\theta&0\\-\delta&1
	\end{pmatrix}$, then equations \eqref{iso_d}, \eqref{iso_e}, and \eqref{iso_f} imply that
	\begin{gather}
		\label{tprime2general}  t'=\begin{pmatrix}
			t_{1}+(1-\lambda_{1})q_{1}\\\delta[t_{1}+(1-\lambda_{1})q_{1}]+\theta[t_{2}+(1-\lambda_{1})q_{2}]
		\end{pmatrix},\\
		\nonumber \mu'=\begin{pmatrix}
			0&0\\-\theta(\lambda_{2}+q_{2})-\delta(\lambda_{1}-\lambda_{4}+q_{1})&\lambda_{1}-\lambda_{4}+q_{1}
		\end{pmatrix},\quad\lambda'=\begin{pmatrix}
			\lambda_{1}&0\\\delta(\lambda_{1}-\lambda_{4}+q_{1})+\theta(\lambda_{2}+q_{2})&\lambda_{4}-q_{1}
		\end{pmatrix}.
	\end{gather}
	Inspecting the above expressions, we see that, by taking $q_{1}=\lambda_{4}-\lambda_{1}$ and $q_{2}=-\lambda_{2}$, it is obtained that $\mu'=0$ and $\lambda'=\operatorname{diag}(\lambda_{1},\lambda_{1}).$
	
	Concerning the associative part, and taking into account that $-\cdot -=0$, equations \eqref{iso_a}, \eqref{iso_b} and \eqref{iso_c} yield
	\begin{gather}
		\nonumber
		s'=\psi(s)\overset{s=0}{=}0,\\
		\label{sigmaleftp2gen}
		\overleftarrow{\sigma'}=\psi\overleftarrow{\sigma}\psi^{-1}=\frac{1}{\theta}\begin{pmatrix}
			\theta\sigma_{1}-\delta\sigma_{3}&\sigma_{3}\\\theta\delta(2\sigma_{1}-1)+\theta^{2}\sigma_{2}-\delta^{2}\sigma_{3}&\delta\sigma_{3}+\theta(1-\sigma_{1})
		\end{pmatrix},\\
		\label{sigmarightp2gen}
		\overrightarrow{\sigma'}=\psi\overrightarrow{\sigma}\psi^{-1}=\frac{1}{\theta}\begin{pmatrix}
			\theta(1-\sigma_{1})+\delta\sigma_{3}&-\sigma_{3}\\
			\theta\delta(1-2\sigma_{1})+\delta^{2}\sigma_{3}-\theta^{2}\sigma_{2}&\theta\sigma_{1}-\delta\sigma_{3}
		\end{pmatrix}.
	\end{gather}
	
	Then, assuming that $\sigma_{3}\neq 0$, one can set $\delta=\theta\sigma_{1}\sigma_{3}^{-1}$. Consequently, \eqref{sigmaleftp2gen} and \eqref{sigmarightp2gen} become 
	\begin{gather*}
		%\label{sigmaleftp2gen2}
		\overleftarrow{\sigma'}=\begin{pmatrix}
			0&\frac{\sigma_{3}}{\theta}\\\frac{\theta}{\sigma_{3}}\left[ \sigma_{2}\sigma_{3}-\sigma_{1}(1-\sigma_{1})\right]&1
		\end{pmatrix}\overset{\sigma_{1}(1-\sigma_{1})=\sigma_{2}\sigma_{3}}{=}\begin{pmatrix}
			0&\frac{\sigma_{3}}{\theta}\\0&1
		\end{pmatrix},
		\\
		\overrightarrow{\sigma'}=\begin{pmatrix}
			1&-\frac{\sigma_{3}}{\theta}\\-\frac{\theta}{\sigma_{3}}\left[\sigma_{2}\sigma_{3}-\sigma_{1}(1-\sigma_{1})\right]&0
		\end{pmatrix}\overset{\sigma_{1}(1-\sigma_{1})=\sigma_{2}\sigma_{3}}{=}\begin{pmatrix}
			1&-\frac{\sigma_{3}}{\theta}\\0&0
		\end{pmatrix}.
	\end{gather*}
	So, setting $\theta=\sigma_{3}$, it leads to $\overleftarrow{\sigma'}=\begin{pmatrix}
		0&1\\0&1
	\end{pmatrix}$ and $\overrightarrow{\sigma'}= \begin{pmatrix}
		1&-1\\0&0
	\end{pmatrix}$. Upon substituting the corresponding values of $q_{1}$, $q_{2}$, $\theta$ and $\delta$ into \eqref{tprime2general}, one finds that 
	\begin{equation}\label{t2primegeneral2}
		t'=\begin{pmatrix}
			t_{1}+(1-\lambda_{1})(\lambda_{4}-\lambda_{1})\\
			\sigma_{1}t_{1}+\sigma_{3}t_{2}+(1-\lambda_{1})[\sigma_{1}(\lambda_{4}-\lambda_{1})-\sigma_{3}\lambda_{2}] 
		\end{pmatrix}.
	\end{equation}
	
	Now, in case that $\sigma_{3}=0$, by the condition $\sigma_{1}(1-\sigma_{1})=\sigma_{2}\sigma_{3}=0$, we get that $\sigma_{1}=0$ or $\sigma_{1}=1$. At first, we consider the case $\sigma_{1}=0$. Substituting the corresponding values into \eqref{sigmaleftp2gen} and \eqref{sigmarightp2gen}, we find that
	\begin{gather*}
		\overleftarrow{\sigma'}=\begin{pmatrix}
			0&0\\\theta\sigma_{2}-\delta&1
		\end{pmatrix},\qquad\overrightarrow{\sigma'}=\begin{pmatrix}
			1&0\\\delta-\theta\sigma_{2}&0
		\end{pmatrix}.
	\end{gather*}
	Hence, setting $\delta=\theta\sigma_{2}$, we obtain that $\overleftarrow{\sigma'}=\begin{pmatrix}
		0&0\\0&1
	\end{pmatrix}$, $\overrightarrow{\sigma'}=\begin{pmatrix}
		1&0\\0&0
	\end{pmatrix}$ and $t'=\begin{pmatrix}
		t_{1}+(1-\lambda_{1})(\lambda_{4}-\lambda_{1})\\\theta\gamma
	\end{pmatrix}$, where $\gamma\coloneqq \sigma_{2}t_{1}+t_{2}+(1-\lambda_{1})[\sigma_{2}(\lambda_{4}-\lambda_{1})-\lambda_{2}]$. Therefore, if $\gamma=0$, then $t'=\begin{pmatrix}
		t_{1}+(1-\lambda_{1})(\lambda_{4}-\lambda_{1})\\0
	\end{pmatrix}$; however, if $\gamma\neq 0$ and $\theta=\gamma^{-1}$, then $t'=\begin{pmatrix}
		t_{1}+(1-\lambda_{1})(\lambda_{4}-\lambda_{1})\\1
	\end{pmatrix}$.
	
	It remains to consider the final case: $\sigma_{3}=0$ and $\sigma_{1}=1$. Substituting these values into \eqref{sigmaleftp2gen} and \eqref{sigmarightp2gen}, it is obtained that
	\[\overleftarrow{\sigma'}=\begin{pmatrix}
		1&0\\\delta+\theta\sigma_{2}&0
	\end{pmatrix},\qquad \overrightarrow{\sigma'}=\begin{pmatrix}
		0&0\\-\delta-\theta\sigma_{2}&1
	\end{pmatrix}.\]
	Thus, by taking $\delta=-\theta\sigma_{2}$, $\overleftarrow{\sigma'}=\begin{pmatrix}
		1&0\\0&0
	\end{pmatrix}$ and $\overrightarrow{\sigma'}=\begin{pmatrix}
		0&0\\0&1
	\end{pmatrix}$ with $t'=\begin{pmatrix}
		t_{1}+(1-\lambda_{1})(\lambda_{4}-\lambda_{1})\\\theta\gamma'
	\end{pmatrix}$, where $\gamma'\coloneqq t_{2}-\sigma_{2}t_{1}+(\lambda_{1}-1)[\sigma_{2}(\lambda_{1}-\lambda_{4})+\lambda_{2}]$. Consequently, when $\gamma'=0$, $t'=\begin{pmatrix}
		t_{1}+(1-\lambda_{1})(\lambda_{4}-\lambda_{1})\\0
	\end{pmatrix}$; although, if $\gamma'\neq 0$, then $t'=\begin{pmatrix}
		t_{1}+(1-\lambda_{1})(\lambda_{4}-\lambda_{1})\\1
	\end{pmatrix}$ setting $\theta=(\gamma')^{-1}$.
	
	The possible isomorphism classes are summarized below:
	{\footnotesize
		\[
		\rotatebox{90}{$\begin{array}{c}\hspace{-1cm}\textnormal{(FE.9)}\\\hspace{-1cm}\paff(P;\sigma, 0;\mu,\lambda,t)\end{array}$}\simeq \left\{ \begin{array}{l c c}
			\paff\left(P;\overleftarrow{\sigma'}=\begin{pmatrix}
				0&1\\0&1
			\end{pmatrix},\overrightarrow{\sigma'}=\begin{pmatrix}
				1&-1\\0&0
			\end{pmatrix}, s'=0;\mu',\lambda',t'=\eqref{t2primegeneral2}\right)&\textnormal{if} &\sigma_{3}\neq 0,\vspace{.1cm}\\
			
			\paff\left(P;\overleftarrow{\sigma'}=\begin{pmatrix}
				0&0\\0&1
			\end{pmatrix},\overrightarrow{\sigma'}=\begin{pmatrix}
				1&0\\0&0
			\end{pmatrix}, s'=0;\mu',\lambda',t'=\begin{pmatrix}
				t_{1}+(1-\lambda_{1})(\lambda_{4}-\lambda_{1})\\0
			\end{pmatrix}\right)&\textnormal{if} &\begin{array}{c}\sigma_{3}= 0=\sigma_{1},\\\gamma=0,\end{array}\vspace{.1cm}\\
			
			\paff\left(P;\overleftarrow{\sigma'}=\begin{pmatrix}
				0&0\\0&1
			\end{pmatrix},\overrightarrow{\sigma'}=\begin{pmatrix}
				1&0\\0&0
			\end{pmatrix}, s'=0;\mu',\lambda',t'=\begin{pmatrix}
				t_{1}+(1-\lambda_{1})(\lambda_{4}-\lambda_{1})\\1
			\end{pmatrix}\right)&\textnormal{if} &\begin{array}{c}\sigma_{3}= 0=\sigma_{1},\\\gamma\neq0,\end{array}\vspace{.1cm}\\
			
			\paff\left(P;\overleftarrow{\sigma'}=\begin{pmatrix}
				1&0\\0&0
			\end{pmatrix},\overrightarrow{\sigma'}=\begin{pmatrix}
				0&0\\0&1
			\end{pmatrix}, s'=0;\mu',\lambda',t'=\begin{pmatrix}
				t_{1}+(1-\lambda_{1})(\lambda_{4}-\lambda_{1})\\0
			\end{pmatrix}\right)&\textnormal{if} &\begin{array}{c}\sigma_{3}= 0;\,\sigma_{1}=1,\\\gamma'=0,\end{array}\vspace{.1cm}\\
			
			\paff\left(P;\overleftarrow{\sigma'}=\begin{pmatrix}
				1&0\\0&0
			\end{pmatrix},\overrightarrow{\sigma'}=\begin{pmatrix}
				0&0\\0&1
			\end{pmatrix}, s'=0;\mu',\lambda',t'=\begin{pmatrix}
				t_{1}+(1-\lambda_{1})(\lambda_{4}-\lambda_{1})\\1
			\end{pmatrix}\right)&\textnormal{if} &\begin{array}{c}\sigma_{3}= 0;\,\sigma_{1}=1,\\\gamma'\neq0,\end{array}
		\end{array}\right.
		\]
	}
	where $\mu'=0$ and $\lambda'=\operatorname{diag}(\lambda_{1},\lambda_{1}).$

	\section*{Funding}
	The research  of Tomasz Brzezi\'nski and Krzysztof Radziszewski is partially supported by the National Science Centre, Poland, through the WEAVE-UNISONO (Grant no.\ 2023/05/Y/ST1/00046). 
	
	The research of Brais Ramos Pérez is partially supported by Ministerio de Ciencia e Innovación, Spain (Grant no. PID2024-155502NB-I00) and Xunta de Galicia (Grant no. ED431C 2023/31).
	
	\bibliographystyle{amsalpha}

\end{document}